\documentclass[11pt]{amsart}

\usepackage{amsmath,amssymb,amsthm}
\usepackage{mathtools}
\usepackage[margin=1in]{geometry}
\usepackage{booktabs}
\usepackage{tikz}
\usepackage{xcolor}

\usetikzlibrary{arrows.meta,calc,decorations.pathreplacing,positioning}
\usepackage{hyperref}

\definecolor{polyfill}{RGB}{219,234,247}
\definecolor{boxblue}{RGB}{31,102,163}
\definecolor{boxgreen}{RGB}{34,139,94}
\definecolor{clipred}{RGB}{200,74,45}
\definecolor{ghost}{RGB}{170,170,170}

\theoremstyle{plain}
\newtheorem{theorem}{Theorem}
\newtheorem{lemma}{Lemma}

\newtheorem{proposition}{Proposition}
\newtheorem{conjecture}{Conjecture}
\theoremstyle{definition}
\newtheorem{definition}{Definition}
\newtheorem{remark}{Remark}

\newcommand{\Z}{\mathbb{Z}}
\newcommand{\R}{\mathbb{R}}
\newcommand{\Wf}{W_{\!f}}
\DeclareMathOperator{\Pt}{\mathcal{P}}
\DeclareMathOperator{\supp}{supp}
\newcommand{\Kf}{\mathcal{K}}

\title{Polytopal invariance for Lusztig's $q$-weight multiplicities}
\author{Nicolás Libedinsky}
\dedicatory{In memory of Papa G., with admiration.}
\address{Departamento de Matemática, Universidad de Chile, Santiago, Chile}
\email{nlibedinsky@gmail.com}

\begin{document}
	\begin{abstract}
		
		Let $\Phi$ be an irreducible root system  normalized so that long 		
	roots have length $\sqrt{2}$, and let $\mu\leq\lambda$ be dominant weights. We associate to $(\Phi,\mu,\lambda)$ the convex polytope
		$$
		P^{\Phi}_{\mu,\lambda}
		=
		(\mu+C)\cap(\lambda-C)\cap(\overline{D}-\rho),
		$$
		where $C=\mathbb{R}_{\geq 0}\Phi^+$ and $\overline{D}$ is the closed dominant Weyl chamber. We prove that if $P^{\Phi}_{\mu,\lambda}$ and $P^{\Phi'}_{\mu',\lambda'}$ are isometric, then
		$$
		K^{\Phi}_{\lambda\mu}(q)
		=
		K^{\Phi'}_{\lambda'\mu'}(q),
		$$
		where $K^{\Phi}_{\lambda\mu}(q)$ is Lusztig's $q$-analogue of the $\mu$-weight multiplicity in the irreducible highest-weight representation $V(\lambda)$, often called the generalized Kostka--Foulkes polynomial. Moreover, we construct the polynomial directly from the metric shape of the polytope, and we classify when two such polytopes are isometric. Such coincidences are abundant and can occur in different Dynkin types, and in different Dynkin families. By the Lusztig--Kato theorem, the main theorem gives a metric analogue, in the spherical setting, of the combinatorial invariance conjecture for Kazhdan--Lusztig polynomials.

	\end{abstract}

	\maketitle
	
	\section{Introduction}\label{sec:intro}

Two of the most basic questions in representation theory are: what is the dimension of an irreducible finite-dimensional representation, and what is its character? For complex semisimple Lie algebras, compact connected Lie groups, and connected reductive algebraic groups over $\mathbb C$, the answers are, rather magically, given by the same formulas.

The first general answer was Weyl's character formula \cite{Weyl}, obtained by analytic and topological methods exactly a century ago. Since then, many formulas for characters and weight multiplicities have emerged from algebraic, combinatorial, and geometric methods \cite{Freudenthal,Kostant,Littelmann,Sahi,Schutzer}. Of particular relevance here are Lusztig's $q$-analogues of weight multiplicities. At $q=1$ they reduce to the usual weight multiplicities and, taken over all weights, recover the character.

Write $\Wf$ for the finite Weyl group and $\rho$ for the half-sum of the positive roots, and let
$\Pt^\Phi_q(\nu)$ denote the $q$-analogue of Kostant's partition function, whose coefficient of
$q^k$ counts the ways of writing $\nu$ as a sum of $k$ positive roots of $\Phi$, with repetitions
allowed and order irrelevant. Lusztig's $q$-analogue of weight multiplicities
\cite[\S9]{Lusztig} is
\begin{equation}\label{eq:lusztig}
	K^{\Phi}_{\lambda\mu}(q)
	=
	\sum_{w\in\Wf}
	(-1)^{\ell(w)}
	\Pt^{\Phi}_q\bigl(w(\lambda+\rho)-(\mu+\rho)\bigr).
\end{equation}
 The specialization $K^\Phi_{\lambda\mu}(1)$ is the ordinary weight multiplicity
$$
\dim V(\lambda)_\mu,
$$
namely the coefficient of $e^\mu$ in the character of the irreducible highest weight representation $V(\lambda)$.

The polynomial $K^\Phi_{\lambda\mu}(q)$ contains subtler information invisible under the specialization $q=1$. For dominant $\mu$, it is the Poincaré polynomial of the Brylinski--Kostant filtration on $V(\lambda)_\mu$, induced by a principal nilpotent element \cite{Brylinski,Broer},    and it encodes the local intersection cohomology of a Schubert variety in the affine Grassmannian of the Langlands dual group \cite[§11]{Lusztig}\cite[§§2.4--2.6]{BravermanFinkelberg}.

We can now state the main result of this paper. Throughout the paper, all root systems are assumed to be reduced. Let $\Phi$ be an irreducible root system in a Euclidean space $(V,(,))$, normalized so that long roots have length $\sqrt{2}$. Let $C=\R_{\geq 0}\Phi^+$ be the positive cone and $\overline{D}$ the closed dominant Weyl chamber. For dominant weights $\mu\leq\lambda$, put
$$
P^\Phi_{\mu,\lambda}=
(\mu+C)\cap(\lambda-C)\cap(\overline{D}-\rho).
$$
By an isometry between two polytopes we mean a distance-preserving bijection for their induced
Euclidean metrics; no rescaling is allowed.
\begin{theorem}[Polytopal invariance]\label{thm:main}
	Let $\Phi,\Phi'$ be irreducible, possibly of different Dynkin types, and let $\mu\leq\lambda$ and $\mu'\leq\lambda'$ be dominant weights. If $P^\Phi_{\mu,\lambda}$ and $P^{\Phi'}_{\mu',\lambda'}$ are isometric, then
	$$
	K^\Phi_{\lambda\mu}(q)
	=
	K^{\Phi'}_{\lambda'\mu'}(q).
	$$
\end{theorem}

The main inspiration for this result is the Combinatorial invariance conjecture for Kazhdan--Lusztig polynomials. Dyer gave the first explicit published formulation \cite{Dyer}. The conjecture is also attributed to Lusztig \cite{Brenti}. This conjecture predicts that a Kazhdan--Lusztig polynomial $h_{x,y}$ depends only on the isomorphism type of the Bruhat interval $[x,y]$. Two natural dreams are:
\begin{enumerate}
	\item[(D1)] to define Kazhdan--Lusztig polynomials directly from the poset structure of $[x,y]$;
	\item[(D2)] to classify Bruhat intervals up to poset isomorphism.
\end{enumerate}
Both problems remain wide open, even in small affine types. In the work \cite{BLV} that directly inspired the present paper, for instance, we obtain only a partial classification of Bruhat intervals in $\widetilde A_2$. For a recent survey and subsequent progress on the conjecture, see \cite{BrentiOpen,BarkleyGaetzLam,CaselliMarietti}. A proof of the conjecture in full generality has very recently been announced in \cite{OpenAICIC}.

In the present paper we replace the isomorphism type of $[x,y]$ by the metric type of $P^\Phi_{\mu,\lambda}$, and ordinary Kazhdan--Lusztig polynomials by spherical Kazhdan--Lusztig polynomials. The connection is provided by the Lusztig--Kato theorem \cite{Lusztig,Kato}. By the Lusztig--Kato theorem, these are, after a change of normalization,
ordinary Kazhdan--Lusztig polynomials for particular pairs of elements
of an affine Weyl group, see \cite[\S2.3]{LPP}. In this metric spherical setting, both dreams admit complete answers.

For (D1), we associate a polynomial $K_P(q)$ directly to the Euclidean shape of every polytope $P$ arising as some $P^\Phi_{\mu,\lambda}$. The construction uses the polytope alone (see Theorem~\ref{prop:KPshape}), and whenever $P$ is isometric to $P^\Phi_{\mu,\lambda}$ one has
$$
K_P(q)=K^\Phi_{\lambda\mu}(q).
$$
Thus a polynomial defined a priori from root-system data can be recovered from the raw Euclidean shape of the polytope alone. The shape does not determine the ambient root system, or even its rank, yet it determines the generalized Kostka--Foulkes polynomial.

For (D2), we call the data $(\Phi,\mu,\lambda)$ producing a given shape a \emph{realization} of it. In Theorem~\ref{cor:class} we give a complete classification of when two realized polytopes are isometric. Roughly speaking, an isometry can arise from a translation, a permutation of simple-root directions, a simultaneous reversal of these directions in the cases where this symmetry is available, or the identification of one parabolic subsystem with another. These phenomena can occur in combination. Moreover, the two realizations may have different Dynkin types, so that the polytopes can live in entirely different Euclidean spaces.

Here is one such coincidence across Dynkin families.\footnote{Here $\varpi_i$ denotes the $i$-th fundamental weight, $\langle\varpi_i,\alpha_j^\vee\rangle=\delta_{ij}$, and simple roots are numbered as in Bourbaki \cite[Planches]{Bourbaki}.} In $D_5$, take
$$
\mu=\varpi_3+\varpi_5,
\qquad
\lambda=\varpi_1+\varpi_2+\varpi_3+\varpi_4,
$$
while in $E_6$ take
$$
\mu'=\varpi_4+\varpi_5+\varpi_6,
\qquad
\lambda'=\varpi_1+\varpi_2+\varpi_3+\varpi_4.
$$
The two polytopes are five-dimensional, both have $68$ vertices, and they are isometric. We have
$$
K^{D_5}_{\lambda\mu}(q)=K^{E_6}_{\lambda'\mu'}(q)
=
q^{2}+5q^{3}+13q^{4}+19q^{5}+20q^{6}+18q^{7}+12q^{8}+6q^{9}+3q^{10}+q^{11}.
$$
This equality is possibly more impressive when viewed from the point of view of Lusztig's formula \eqref{eq:lusztig}: the two expressions run over Weyl groups of orders $1{,}920$ and $51{,}840$, and use root systems with $20$ and $36$ positive roots, respectively.

These coincidences are quite common. 
 An exhaustive computation over bounded
sets of dominant pairs in types $A_1$--$A_5$, $B_2$--$B_4$, $C_3$, $C_4$,
$D_4$, $D_5$, $E_6$, $F_4$ and $G_2$\footnote{The fundamental-weight
	coordinates were bounded by $12$ in ranks one and two, $8$ in $A_3$,
	$6$ in $B_3$ and $C_3$, $3$ in rank four, and $1$ in $A_5$, $D_5$
	and $E_6$.}
produced $212{,}441$ nonempty polytopes in $113{,}620$ isometry classes.
Of these, $18{,}630$ contain more than one realization, $1{,}517$ are realized in more than one Dynkin type, and $1{,}010$ bridge two different Dynkin families. The code and an independent verification of the
polynomial identities are available at \cite{Code}.

 The proof explains why these coincidences are so abundant: it isolates exactly what part of the root-system data is visible to Lusztig's formula, and shows that this is precisely what the metric shape remembers. On the algebraic side, we show that the polynomial depends only on a parabolic subsystem, together with a small amount of numerical data. On the geometric side, we prove that precisely this data can be reconstructed from the Euclidean shape of the polytope itself.  Combining the two gives Theorem \ref{thm:main}.

Through the Lusztig--Kato theorem, Theorem~\ref{thm:main} is a statement about Kazhdan--Lusztig polynomials $h_{x,y}$ for particular pairs $x\le y$ in affine Weyl groups. In Section~\ref{sec:hecke} we conjecture that for these pairs the isometry class of the polytope determines not only $h_{x,y}$ but the isomorphism class of the Bruhat interval $[x,y]$, even across different Dynkin types. This raises a wider question: how much of affine Kazhdan--Lusztig theory can be understood with metric geometry?

\subsection*{Acknowledgments}
The polytope $P^{\Phi}_{\mu,\lambda}$ and several ideas underlying this paper first emerged in $\widetilde{A}_2$ in joint work \cite{BLV} with Rodrigo Villegas and Gastón Burrull, whom I am very happy to thank. I also benefited greatly from discussions with Damián de la Fuente, David Plaza, and Federico Castillo on the interplay between alcovic and Euclidean geometry \cite{BOAT}. David's comments, in particular, led me to substantially reformulate the paper. 
 
 \subsection*{Use of artificial intelligence}
 
 The author used Claude Fable 5 (Anthropic) in the course of this work. The author formulated the statement of Theorem 1 as a conjecture, and large-scale computations by Fable confirmed it in all Dynkin types of small rank that were tested (see the census above). Once the statement was supported by this evidence, the author identified the four mechanisms producing isometries between the polytopes (two of which appear in his previous work \cite{BLV}), and expected the proof to proceed by reducing Lusztig's formula to a parabolic subsystem and recovering that subsystem from the shape of the polytope, using the facet normals with and without antipodal partner. Given little more than the statement, Fable independently produced a three-page argument proving it along essentially these lines. The author subsequently rewrote and expanded that argument into the present paper of more than twenty pages, correcting errors, supplying details, and developing its current conceptual approach: constructing the polynomial directly from the polytope and classifying the relevant isometries completely.  Fable also wrote most of the code available at \cite{Code}. The author is solely responsible for the content of the final manuscript.


\section{Recovering the polynomial from the shape}\label{sec:formula}

In this section $\Phi$ is allowed to be reducible and is reduced (as throughout) and is assumed to span its ambient Euclidean space. Each irreducible component is normalized so that its long roots have length $\sqrt2$ (this normalization plays no role in this section: only coroot pairings enter), and $K^\Phi_{\lambda\mu}(q)$ is defined by \eqref{eq:lusztig}.

Fix a base $\Delta=\{\alpha_1,\dots,\alpha_n\}$ of $\Phi$, where $n$ is the rank of $\Phi$,
and write $\Phi^+$ for the corresponding set of positive roots and $(\,,\,)$ for the Euclidean
inner product on $V$. For $1\le i\le n$, define the simple coroot by
$\alpha_i^\vee:=2\alpha_i/(\alpha_i,\alpha_i)$ and set
$D_i:=(\alpha_i,\alpha_i)/2\in\{1,\tfrac12,\tfrac13\}$, so that
$\alpha_i^\vee=\alpha_i/D_i$. We define
$\langle v,\alpha_i^\vee\rangle:=2(v,\alpha_i)/(\alpha_i,\alpha_i)
=(v,\alpha_i)/D_i$. Thus $\langle\,\cdot\,,\cdot\,\rangle$ is not a new pairing: it is the inner product $(\,\cdot\,,\cdot\,)$ itself, with the second argument a coroot or coweight regarded as a vector of $V$. We keep the angle-bracket notation for pairings whose values are integers independent of the normalization, and reserve $(\,\cdot\,,\cdot\,)$ for the metric quantities (lengths, distances, Gram matrices) that the isometry hypothesis controls.

For any root $\gamma$, let
$$
s_\gamma(v):=v-\frac{2(v,\gamma)}{(\gamma,\gamma)}\,\gamma
$$
be the corresponding reflection. In particular, for the simple root $\alpha_i$ we write
$s_i:=s_{\alpha_i}$, so
$$
s_i(v)=v-\langle v,\alpha_i^\vee\rangle\alpha_i\text{,}
$$
and we set $S:=\{s_1,\dots,s_n\}$. These reflections generate $\Wf$. For $w\in\Wf$, write $\ell(w)$ for its length with respect to $S$.

Fix $\Phi$ and dominant weights $\mu\le\lambda$, where $\le$ denotes the dominance order:
$\lambda-\mu\in\Z_{\ge0}\Delta$. Write $\beta=\lambda-\mu=\sum_k m_k\alpha_k$, so that $M=(m_1,\dots,m_n)\in\Z_{\ge0}^n$. The numbers $m_i$ will be called the \emph{widths}. Write $\hat\lambda=\lambda+\rho$ and $\hat\mu=\mu+\rho$, so that $\beta=\hat\lambda-\hat\mu$. For any weight $\eta$, write
$\eta_i:=\langle\eta,\alpha_i^\vee\rangle$
for its $i$-th fundamental-weight coordinate. One has $\hat\lambda_i-\hat\mu_i=\langle\beta,\alpha_i^\vee\rangle$. In this notation
\begin{equation}\label{lambdamu}
 \hat\mu_i=\langle\mu,\alpha_i^\vee\rangle+1\ge1,\qquad
\hat\lambda_i=\langle\lambda,\alpha_i^\vee\rangle+1\ge1,
\end{equation}  
both identities holding because the half-sum of the positive roots of a root system pairs to $1$
with each of its simple coroots \cite[Ch.~VI, \S1.10]{Bourbaki}. By \eqref{lambdamu}, the weights $\hat\mu$ and $\hat\lambda$ are regular dominant. Define the \emph{parallelepiped}
\[
  \Pi \;=\; \Pi^{\Phi}_{\mu,\lambda}\;\coloneqq\; \mu + \sum_{k=1}^{n} \,[0,m_k]\,\alpha_k.
\]
The parallelepiped can also be written as $\Pi=(\mu+C)\cap(\lambda-C)$: if one writes
$v=\mu+\sum_k a_k\alpha_k$, one has that $v\in\mu+C$ means that every $a_k\ge0$, while $v\in\lambda-C$
means that every $m_k-a_k\ge0$, so $0\le a_k\le m_k$.

It follows  that $\Pi$ contains $P\coloneqq P^{\Phi}_{\mu,\lambda}$, and
the latter polytope is the parallelepiped cut by the shifted walls:
\[P=\Pi\cap\bigcap_i\{v:\langle v,\alpha_i^\vee\rangle\ge -1\}.\]

We will frequently write $v=\mu+\sum_k a_k\alpha_k$ for $v\in\Pi$, and occasionally use
$a_k(v)$ to emphasize the dependence on $v$. For each simple index $i$, define the signed
\emph{wall depth}
$$
\delta_i:=-1-\min_{v\in\Pi}\langle v,\alpha_i^\vee\rangle,
$$
and write $\delta:=(\delta_1,\dots,\delta_n)$.
On $\Pi$ one has
$\langle v,\alpha_i^\vee\rangle=(\hat\mu_i-1)+\sum_k a_k\langle\alpha_k,\alpha_i^\vee\rangle$.
Since $\langle\alpha_i,\alpha_i^\vee\rangle=2>0$ and
$\langle\alpha_k,\alpha_i^\vee\rangle\le0$ for $k\ne i$, the minimum is attained at
$a_i=0$ and $a_k=m_k$ for $k\ne i$. Thus
\begin{equation}\label{deltai}
\delta_i=2m_i-\langle\beta,\alpha_i^\vee\rangle-\hat\mu_i=2m_i-\hat\lambda_i.
\end{equation}
 The wall meets $\Pi$ when
$\delta_i\ge0$, but cuts out a corner only when $\delta_i>0$: at $\delta_i=0$ it merely touches
$\Pi$ and removes nothing. When $\delta_i>0$, the removed corner has depth $\delta_i$, as measured
by the coroot pairing $\langle\,\cdot\,,\alpha_i^\vee\rangle$.

\begin{figure}[ht]
	\centering
	\begin{tikzpicture}[scale=1.2,>={Stealth[length=6pt]},line join=round]
		\coordinate (O)   at (0,0);
		\coordinate (P1)  at (2.8284,0);
		\coordinate (P12) at (1.4142,2.4495);
		\coordinate (P2)  at (-1.4142,2.4495);
		\coordinate (h2) at (1.4142,0);
		\coordinate (h3) at (2.4749,0.6124);
		\coordinate (h5) at (-0.7071,2.4495);
		\coordinate (h6) at (-0.7071,1.2247);
		\fill[ghost!22] (h2)--(P1)--(h3)--cycle;
		\fill[ghost!22] (h5)--(P2)--(h6)--cycle;
		\draw[ghost,dashed] (O)--(P1)--(P12)--(P2)--cycle;
		\fill[ghost] (P1) circle (1.1pt);
		\fill[ghost] (P2) circle (1.1pt);
		\node[ghost!150] at (1.62,1.72) {$\Pi$};
		\fill[polyfill] (O)--(h2)--(h3)--(P12)--(h5)--(h6)--cycle;
		\draw[boxblue,very thick] (O)--(h2)--(h3)--(P12)--(h5)--(h6)--cycle;
		\draw[clipred,very thick] (h2)--(h3);
		\draw[clipred,very thick] (h5)--(h6);
		\draw[->,thin] (P1)--(h3);
		\node[right=1pt,font=\small] at ($(P1)!0.5!(h3)$) {$\delta_2\tfrac{\alpha_2}{2}$};
		\draw[->,thin] (P2)--(h5);
		\node[above,font=\small] at ($(P2)!0.5!(h5)$) {$\delta_1\tfrac{\alpha_1}{2}$};
		\draw[decorate,decoration={brace,amplitude=4pt,mirror},ghost!160]
		($(O)+(0,-0.12)$)--($(P1)+(0,-0.12)$);
		\node[below,font=\small] at (1.4142,-0.34) {$m_1\,\alpha_1$};
		\draw[decorate,decoration={brace,amplitude=4pt},ghost!160]
		($(O)+(-0.10,0.05)$)--($(P2)+(-0.10,0.05)$);
		\node[left,font=\small] at (-1.15,1.5) {$m_2\,\alpha_2$};
		\draw[->,boxblue,very thick] (O)--(h2);
		\node[boxblue,font=\small] at (0.75,0.28) {$\alpha_1$};
		\draw[->,boxgreen,very thick] (O)--(h6);
		\node[boxgreen,font=\small] at (-0.12,0.78) {$\alpha_2$};
		\fill[black] (O) circle (1.6pt);
		\node[below left=0pt] at (O) {$\mu$};
		\node[above] at (P12) {$\lambda$};
		\node[clipred,font=\small] at (3.55,1.15) {$\langle v,\alpha_2^\vee\rangle=-1$};
		\draw[clipred,thin] (3.05,0.95)--($(h2)!0.55!(h3)$);
		\draw[clipred,thin] (h6)--(-0.7071,0.35);
		\node[clipred,font=\small,left] at (-0.7071,0.42) {$\langle v,\alpha_1^\vee\rangle=-1$};
	\end{tikzpicture}
	\caption{The polytope $P$ (light blue) is the parallelepiped $\Pi$ 
		cut by the shifted chamber walls $\langle v,\alpha_i^\vee\rangle=-1$ (red). The edges of $\Pi$ leave the
		base vertex $\mu$ in the simple-root directions $\alpha_1,\alpha_2$. The edge vectors are $m_i\alpha_i$. The wall depth $\delta_i$ measures how far
		the $i$-th wall cuts into $\Pi$: it cuts off the vector $\delta_i\tfrac{\alpha_i}{2}$ from the corresponding edge of $\Pi$, as marked. Here $\Phi=A_2$, $\mu=0$,
		$m_1=m_2=2$ and $\delta_1=\delta_2=1$.}
	\label{fig:reconstruction}
\end{figure}
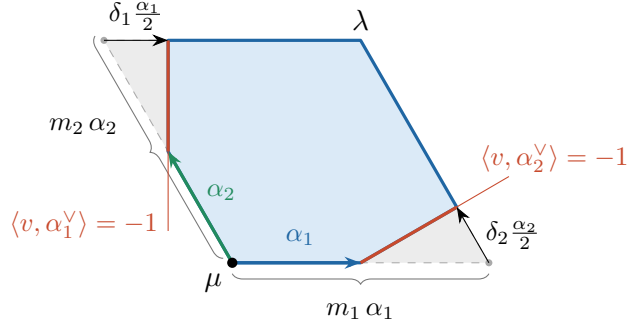

Define the set of \emph{active walls}
\[
  J\coloneqq\{\,i:\delta_i\ge m_i\,\},
\]
and note that a wall may cut $\Pi$ ($\delta_i>0$) without being active (Figure~\ref{fig:active}).
Call the wall $i$ \emph{cutting} if $\delta_i\ge1$, let $\mathrm{Cut}\coloneqq\{i:\delta_i\ge1\}$ be
the set of cutting walls, and let $I\coloneqq\supp(M)=\{i:m_i>0\}$ be the \emph{support} of $M$. The
three sets are nested.

\begin{lemma}\label{lem:JCutI}
\[
\underbrace{\{\,i:\delta_i\ge m_i\,\}}_{J\ \text{(active)}}
\;\subseteq\;
\underbrace{\{\,i:\delta_i\ge 1\,\}}_{\mathrm{Cut}\ \text{(cutting)}}
\;\subseteq\;
\underbrace{\{\,i:m_i> 0\,\}}_{I\ \text{(support)}} .
\]
\end{lemma}

\begin{proof}
If $i\notin I$ then $m_i=0$, so by \eqref{deltai} $\delta_i=-\hat\lambda_i\le-1$ by \eqref{lambdamu}; such an $i$ is neither
cutting nor active, which gives both $\mathrm{Cut}\subseteq I$ and $J\subseteq I$. If now $i\in J$
then $i\in I$ by what we have just seen, so $m_i\ge1$ and therefore $\delta_i\ge m_i\ge1$, thus the wall
$i$ is cutting.
\end{proof}

The following elementary remark will be used repeatedly below.
\begin{remark}\label{elrem}
	 For $\nu=\sum_k b_k\alpha_k$ with
	$b_k\in\Z$, one has $\Pt^{\Phi}_q(\nu)\ne0$ if and only if every $b_k\ge0$. Indeed positive roots
	have nonnegative coordinates in the base, and conversely the simple roots are themselves positive
	roots.
\end{remark}

The threshold $\delta_i\ge m_i$ has both a geometric and an algebraic meaning. Geometrically,
let $E_i$ be the edge of $\Pi$ starting at $\lambda$ in the direction $-\alpha_i$. Its other
endpoint is the vertex $v_i^{\min}$, where $a_i=0$ and $a_k=m_k$ for $k\ne i$, and
$\langle\,\cdot\,,\alpha_i^\vee\rangle$ attains its minimum on $\Pi$. Along $E_i$ we may write $v=v_i^{\min}+a_i(v)\alpha_i$. Since
$\langle v_i^{\min},\alpha_i^\vee\rangle=-1-\delta_i$, we have
$$
\langle v,\alpha_i^\vee\rangle=-1-\delta_i+2a_i(v).
$$
Thus the wall $\langle v,\alpha_i^\vee\rangle=-1$ meets $E_i$ at
$a_i(v)=\delta_i/2$, while the midpoint has $a_i(v)=m_i/2$. Then $i$ is active precisely when
the wall cuts at least halfway across $E_i$; an \emph{inert cut}
($0<\delta_i<m_i$) cuts it before the midpoint.

Algebraically, the term of Lusztig's formula corresponding to the simple reflection $s_i$
is nonzero exactly when $i$ is active. Indeed, since
$s_i\hat\lambda=\hat\lambda-\hat\lambda_i\alpha_i$, by \eqref{deltai} we have
$$
s_i\hat\lambda-\hat\mu
=\beta-\hat\lambda_i\alpha_i
=(\delta_i-m_i)\alpha_i+\sum_{k\ne i}m_k\alpha_k.
$$
All coefficients except possibly that of $\alpha_i$ are nonnegative, and the coefficient of
$\alpha_i$ is $\delta_i-m_i$. By Remark~\ref{elrem},
$\Pt^\Phi_q(s_i\hat\lambda-\hat\mu)\ne0$ if and only if $\delta_i\ge m_i$.
Proposition~\ref{prop:collapse} extends this observation to all the terms of Lusztig's formula.

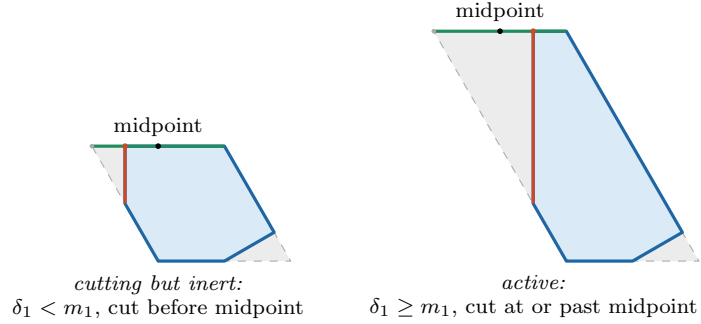
\begin{figure}[ht]
	\centering
	\begin{tikzpicture}[scale=1.0,>={Stealth[length=5pt]},line join=round]
		\begin{scope}[scale=0.62]
			\coordinate (O)  at (0,0);
			\coordinate (P1) at (2.8284,0);
			\coordinate (P12)at (1.4142,2.4495);
			\coordinate (P2) at (-1.4142,2.4495);
			\coordinate (h2) at (1.4142,0);
			\coordinate (h3) at (2.4749,0.6124);
			\coordinate (h5) at (-0.7071,2.4495);
			\coordinate (h6) at (-0.7071,1.2247);
			\coordinate (mid)at (0,2.4495);
			\fill[ghost!22] (h5)--(P2)--(h6)--cycle;
			\fill[ghost!22] (h2)--(P1)--(h3)--cycle;
			\draw[ghost,dashed] (O)--(P1)--(P12)--(P2)--cycle;
			\fill[polyfill] (O)--(h2)--(h3)--(P12)--(h5)--(h6)--cycle;
			\draw[boxblue,very thick] (O)--(h2)--(h3)--(P12)--(h5)--(h6)--cycle;
			\draw[boxgreen,very thick] (P2)--(P12);
			\draw[clipred,very thick] (h5)--(h6);
			\fill[ghost] (P2) circle (1.4pt);
			\fill[black] (mid) circle (1.7pt);
			\fill[clipred] (h5) circle (1.7pt);
			\node[above,font=\scriptsize] at (mid) {midpoint};
			\node[font=\scriptsize,align=center] at (0,-0.75) {\emph{cutting but inert:}\\[-1pt]$\delta_1<m_1$, cut before midpoint};
		\end{scope}
		\begin{scope}[xshift=5.4cm,scale=0.62]
			\coordinate (aO)  at (0,0);
			\coordinate (aP1) at (2.8284,0);
			\coordinate (aP12)at (0,4.8990);
			\coordinate (aP2) at (-2.8284,4.8990);
			\coordinate (a1) at (1.4142,0);
			\coordinate (a2) at (2.4749,0.6124);
			\coordinate (a4) at (-0.7071,4.8990);
			\coordinate (a5) at (-0.7071,1.2247);
			\coordinate (amid) at (-1.4142,4.8990);
			\fill[ghost!22] (a4)--(aP2)--(a5)--cycle;
			\fill[ghost!22] (a1)--(aP1)--(a2)--cycle;
			\draw[ghost,dashed] (aO)--(aP1)--(aP12)--(aP2)--cycle;
			\fill[polyfill] (aO)--(a1)--(a2)--(aP12)--(a4)--(a5)--cycle;
			\draw[boxblue,very thick] (aO)--(a1)--(a2)--(aP12)--(a4)--(a5)--cycle;
			\draw[boxgreen,very thick] (aP2)--(aP12);
			\draw[clipred,very thick] (a4)--(a5);
			\fill[ghost] (aP2) circle (1.4pt);
			\fill[black] (amid) circle (1.7pt);
			\fill[clipred] (a4) circle (1.7pt);
			\node[above,font=\scriptsize] at (amid) {midpoint};
			\node[font=\scriptsize,align=center] at (-0.7,-0.75) {\emph{active:}\\[-1pt]$\delta_1\ge m_1$, cut at or past midpoint};
		\end{scope}
	\end{tikzpicture}
\caption{Inert versus active cuts in $A_2$. Left: $\delta_1<m_1$, so the wall cuts
	the marked edge before its midpoint. Right: $\delta_1\ge m_1$, so it cuts it at or past the midpoint.}
	\label{fig:active}
\end{figure}

Let $W_J\le\Wf$ be the subgroup generated by $\{s_j:j\in J\}\subseteq S$ (similarly, for a subset $K\subseteq S$ we write $W_K$ for the subgroup generated by $K$), and for $w\in W_J$ put
\[
  v_w\;\coloneqq\;w\hat\lambda-\hat\mu .
\]
 The group $W_J$ is a Coxeter group with simple generators $\{s_j:j\in J\}$, and its length
 function is the restriction of $\ell$ \cite[\S5.5]{HumphreysRGCG}.

Write $\varpi_i^\vee$ for the fundamental coweight dual to $\alpha_i$:
it is determined by $\langle\alpha_j,\varpi_i^\vee\rangle=\delta_{ij}$.
Using the equivalent form
$$
s_j(v)=v-(\alpha_j,v)\alpha_j^\vee
$$
of the reflection formula, we have
$s_j\varpi_i^\vee=\varpi_i^\vee$ for $j\ne i$; hence $\varpi_i^\vee$ is fixed by the parabolic subgroup $W_{S\smallsetminus\{s_i\}}$, and
in fact this parabolic is exactly the stabilizer of $\varpi_i^\vee$ in $\Wf$, by the standard
description of the stabilizer of a point of the closed chamber \cite[Ch.~V, \S3, no.~3, Prop.~2]{Bourbaki}.

\begin{proposition}\label{prop:collapse}
With $J=\{i:\delta_i\ge m_i\}$ the active set,
\[
  K^{\Phi}_{\lambda\mu}(q)=\sum_{w\in W_{J}}(-1)^{\ell(w)}\Pt^{\Phi}_q(w\hat\lambda-\hat\mu).
\]
\end{proposition}

\begin{proof}

Fix $w\in\Wf$ with reduced expression $s_{j_1}\cdots s_{j_\ell}$, and set
$u_1:=\mathrm{id}$ and $u_k:=s_{j_1}\cdots s_{j_{k-1}}$ for $2\le k\le\ell+1$. 
The telescoping identity
\begin{equation}\label{eq:telescope}
  \hat\lambda-w\hat\lambda=\sum_{k=1}^{\ell}\langle\hat\lambda,\alpha_{j_k}^\vee\rangle\,u_k(\alpha_{j_k})
\end{equation}
is obtained by expanding $\hat\lambda-w\hat\lambda=\sum_{k=1}^{\ell}u_{k}(\hat\lambda-s_{j_k}\hat\lambda)$ and using
$\hat\lambda-s_{j_k}\hat\lambda=\langle\hat\lambda,\alpha_{j_k}^\vee\rangle\alpha_{j_k}$. Each $u_k(\alpha_{j_k})$ is a positive root: since
$u_ks_{j_k}=s_{j_1}\cdots s_{j_k}$ is reduced, this follows from
\cite[Ch.~VI, \S1, no.~6, Cor.~2 to Prop.~17]{Bourbaki}.

We now show that every nonvanishing term of Lusztig's sum fixes  $\varpi_i^\vee$ for all $i\notin J.$

Suppose $w\varpi_i^\vee\ne\varpi_i^\vee$, equivalently
$w\notin$ $W_{S\smallsetminus\{s_i\}}$. Then some letter in the chosen reduced
expression of $w$ is $s_i$. For the first such occurrence, say $j_k=i$, we have
$u_k\in$ $W_{S\smallsetminus\{s_i\}}$, so $u_k$ fixes $\varpi_i^\vee$.
Pairing \eqref{eq:telescope} with $\varpi_i^\vee$, every summand is nonnegative,
since $u_k(\alpha_{j_k})$ is a positive root and positive roots pair
nonnegatively with $\varpi_i^\vee$, while
$\langle\hat\lambda,\alpha_{j_k}^\vee\rangle>0$ by \eqref{lambdamu}. For the first
occurrence of $i$,
\[
\langle u_k(\alpha_i),\varpi_i^\vee\rangle
=\langle u_k(\alpha_i),u_k(\varpi_i^\vee)\rangle
=\langle\alpha_i,\varpi_i^\vee\rangle
=1.
\]
So we have
\[
\langle\hat\lambda-w\hat\lambda,\varpi_i^\vee\rangle
\ge \langle\hat\lambda,\alpha_i^\vee\rangle
=\hat\lambda_i.
\]

Now assume in addition that
$\Pt^\Phi_q(w\hat\lambda-\hat\mu)\ne0$. Then
$w\hat\lambda-\hat\mu\in\Z_{\ge0}\Phi^+$, and so
$\langle w\hat\lambda-\hat\mu,\varpi_i^\vee\rangle\ge0$. Therefore
\[
0
\le \langle w\hat\lambda-\hat\mu,\varpi_i^\vee\rangle
= \langle\hat\lambda-\hat\mu,\varpi_i^\vee\rangle
-\langle\hat\lambda-w\hat\lambda,\varpi_i^\vee\rangle
\le m_i-\hat\lambda_i,
\]
where
$\langle\hat\lambda-\hat\mu,\varpi_i^\vee\rangle
=\langle\beta,\varpi_i^\vee\rangle=m_i$.
This together with \eqref{deltai} gives
$\delta_i=2m_i-\hat\lambda_i\ge m_i$, so $i\in J$. We then conclude that if
$\Pt^\Phi_q(w\hat\lambda-\hat\mu)\ne0$ and $i\notin J,$ then
$w\varpi_i^\vee=\varpi_i^\vee$. Consequently, every nonvanishing term fixes
\[
x:=\sum_{i\notin J}\varpi_i^\vee.
\]
This point lies in the closed dominant chamber and satisfies
$\langle\alpha_j,x\rangle=0$ exactly for $j\in J$, so
$\mathrm{Stab}_{\Wf}(x)=W_J$ by the description of stabilizers recalled above.
Thus every nonvanishing term has $w\in W_J$, 
as claimed.

\end{proof}

To put Proposition~\ref{prop:collapse} in context,  the support of
Kostant's formula is known as the Weyl alternation set \cite{Harris} and is an order ideal in both
the left and right weak Bruhat orders \cite[Theorem~3.1]{Weylalt}.

\begin{proposition}\label{prop:KP}
Let $w\in W_J$, let $s_{j_1}\cdots s_{j_r}$ be a reduced expression for $w$ and put
$u_k=s_{j_1}\cdots s_{j_{k-1}}$ as before. Then
\begin{equation}\label{eq:vw}
  v_w\;=\;\beta-\sum_{k=1}^{r}\hat\lambda_{j_k}\,u_k(\alpha_{j_k}).
\end{equation}
Consequently, Proposition~\ref{prop:collapse} takes the form
\begin{equation}\label{eq:KP}
  K^{\Phi}_{\lambda\mu}(q)\;=\;\Kf_\Phi(M,\delta)(q)\;\coloneqq \;\sum_{w\in W_J}(-1)^{\ell(w)}\,\Pt^{\Phi}_q(v_w).
\end{equation}

In particular, for fixed $\Phi$, the right-hand side of \eqref{eq:KP} depends only on
$(M,\delta)$, and in fact on $\delta$ only through its restriction
$\delta|_{\mathrm{Cut}}$.

\end{proposition}

\begin{proof}
Since $\beta=\hat\lambda-\hat\mu$ and $\hat\lambda_{j_k}=\langle\hat\lambda,\alpha_{j_k}^\vee\rangle$, the
telescoping identity \eqref{eq:telescope} gives
\[
  v_w=w\hat\lambda-\hat\mu=\beta-(\hat\lambda-w\hat\lambda)
  =\beta-\sum_{k=1}^{r}\hat\lambda_{j_k}\,u_k(\alpha_{j_k}),
\]
which is \eqref{eq:vw}. Although the right-hand side of \eqref{eq:vw} is written using a reduced expression of $w$,
it is independent of this choice, since it equals the intrinsically defined vector
$v_w=w\hat\lambda-\hat\mu$. Formula \eqref{eq:KP} is now Proposition~\ref{prop:collapse} read through the
definition of $v_w$.

For the last assertion, fix $\Phi$. Then the root system itself determines the simple roots
$\alpha_i$, the Weyl group and its length function, and the partition function $\Pt^\Phi_q$.
We compute \eqref{eq:vw} with any reduced expression of $w$; its letters are all among the $s_j$, $j\in J$, since every reduced expression of an element of $W_J$ only involves the generators $s_j$, $j\in J$ \cite[\S5.5]{HumphreysRGCG}. The remaining data entering the right-hand side of \eqref{eq:KP} are
$\beta=\sum_i m_i\alpha_i$, the set $J=\{i:\delta_i\ge m_i\}$, and the integers
$\hat\lambda_j=2m_j-\delta_j$ for $j\in J$. These are determined by $(M,\delta)$.
Moreover, since $J\subseteq\mathrm{Cut}$, both $J$ and the integers $\hat\lambda_j$ for
$j\in J$ are already determined by $M$ and $\delta|_{\mathrm{Cut}}$.

\end{proof}

\section{The parabolic subsystem \texorpdfstring{$\Phi_I$}{Phi\_I}}\label{sec:parab}

The formula $\Kf_\Phi(M,\delta)$ of Section~\ref{sec:formula} still carries the ambient root system
$\Phi$, through $W_J$ and $\Pt^{\Phi}_q$. Theorem~\ref{thm:parab} below shows that only the
parabolic subsystem $\Phi_I$ spanned by the edges of the parallelepiped is relevant: part~(1)
computes the formula entirely inside $\Phi_I$, while part~(2) gives the transport statement used
in the proof of Theorem~\ref{thm:main}. Since the normalization is not used in Section~\ref{sec:formula}, its results hold for any reduced root system. We apply them to $\Phi_I$ with the metric inherited from $V$ (not renormalized), which is reducible whenever $I$ is
disconnected.\footnote{A different form of independence from the ambient root system appears in
		\cite{LecouveyStab}, where, for fixed partitions $\lambda,\mu$ in the classical families,
		each coefficient of $K^\Phi_{\lambda\mu}(q)$ stabilizes as the ambient rank tends to infinity.}

For a polytope $P$, we call the triple $(\Phi,\mu,\lambda)$ a \emph{realization} of $P$ if $P=P^\Phi_{\mu,\lambda}$, while the \emph{raw shape} of $P$ is $P$ regarded as a metric space with its induced Euclidean metric.
Note that $\mu$ and $\lambda$ always lie in $P$: indeed $\lambda\in\mu+C$ and $\mu\in\lambda-C$
because $\beta\in C$, while $\langle\mu,\alpha_i^\vee\rangle\ge0\ge-1$ and
$\langle\lambda,\alpha_i^\vee\rangle\ge0\ge-1$ for every $i$. Then $P$ is a singleton if and
only if $\mu=\lambda$, if and only if $I=\varnothing$, and in that case $K^{\Phi}_{\lambda\mu}=1$.
Since a polytope isometric to a single point is itself a single point, Theorem~\ref{thm:main} holds
trivially in this case, and we may assume $I\ne\varnothing$ from now on. Let
$I=I_1\sqcup\cdots\sqcup I_r$ be the partition of $I$ into the connected
components of the subdiagram of the Dynkin diagram of $\Phi$ induced by $I$, and let
\[
  U\coloneqq \operatorname{span}\{\alpha_i:i\in I\}=\textstyle\bigoplus_{c}U_c,\qquad \mathrm{with}\ 
  U_c=\operatorname{span}\{\alpha_i: i\in I_c\},
\]     
an orthogonal decomposition. Let $\Phi_I=\Phi\cap
U$ be the corresponding parabolic root subsystem; the set $\{\alpha_i:i\in I\}$ is a base of $\Phi_I$ (i.e.\ a
system of simple roots for it). For $v\in \Pi,$ if one writes $v-\mu=\sum_k a_k\alpha_k$, then $a_k\in[0,m_k]$. In particular,
$a_k=0$ for $k\notin I$, so $v-\mu=\sum_{i\in I}a_i\alpha_i$ and $\Pi-\mu\subseteq U$, thus $P-\mu\subseteq U$. In particular,
\begin{equation}\label{beta}
\beta=\sum_{i\in I}m_i\alpha_i.
\end{equation}
Recall from Lemma~\ref{lem:JCutI} the nested sets $J\subseteq\mathrm{Cut}\subseteq I$.   For each $k\in\{1,\dots,n\}$ and $v\in\mu+U$ written as
$v=\mu+\sum_{i\in I}a_i\alpha_i$, let
\begin{equation}\label{tau}
\tau_k(v):=\langle v,\alpha_k^\vee\rangle+1
=\sum_{i\in I}\langle\alpha_i,\alpha_k^\vee\rangle a_i+\hat\mu_k
\end{equation}

be the \emph{signed margin} of $v$ relative to the $k$-th wall: it is positive on the
dominant side, zero on the wall, and negative beyond it.
Since $\min_{\Pi}\tau_k=-\delta_k$, if $k\notin\mathrm{Cut}$, then
$\tau_k(v)\ge-\delta_k\ge0$ for every $v\in\Pi$. Thus the $k$-th wall inequality
is satisfied throughout $\Pi$ and imposes no additional restriction. Thus
\[
P=\Pi\cap\bigcap_{k\in\mathrm{Cut}}\{\tau_k\ge0\}.
\]
where by $\{\tau_k\ge 0\}$ we mean $\{v\in \mu+U\, \vert\, \tau_k(v)\ge0\}$. For $i\in I$ let $\tilde\varpi_i\in U$ denote the fundamental weights of the subsystem $\Phi_I$: the basis
of $U$ dual to the simple coroots of $\Phi_I$, characterized by
$\langle\tilde\varpi_i,\alpha_j^\vee\rangle=\delta_{ij}$ for all $j\in I$. Two reformulations are used
below. The defining relations read
$(\tilde\varpi_i,\alpha_j)=D_j\,\delta_{ij}$ in terms of the inner product;
and for $v\in\mu+U$, by applying $(-,\tilde\varpi_i)$ one obtains the $i$-th parallelepiped coordinate
\begin{equation}\label{ai}
	a_i(v)=(v-\mu,\tilde\varpi_i)/D_i.
\end{equation}
We use the weights $\tilde\varpi_i$ of $\Phi_I$, rather than the ambient fundamental weights
$\varpi_i$ of $\Phi$, because the latter need not lie in $U$.  The theorem below uses a metrically strict notion of isomorphism between parabolic root
subsystems, which we now define. Recall that $\{\alpha_i:i\in I\}$ is a base of $\Phi_I$.

\begin{definition}\label{def:basediso}
	Let $\Psi$ and $\Psi'$ be root systems with chosen bases $\Delta$ and
	$\Delta'$, spanning Euclidean spaces $E$ and $E'$. A \emph{based isometry}
	$\sigma:(\Psi,\Delta)\to(\Psi',\Delta')$ is a bijection
	$\sigma:\Delta\to\Delta'$ whose linear extension $E\to E'$ is an isometry.
Equivalently, the Gram matrices
\[
G_\Delta:=\big((\alpha,\gamma)\big)_{\alpha,\gamma\in\Delta}
\quad\text{and}\quad
G_{\Delta'}:=\big((\alpha',\gamma')\big)_{\alpha',\gamma'\in\Delta'}
\]
agree under $\sigma$.
\end{definition}

The usual notion of isomorphism of root systems allows each irreducible factor to be rescaled; here
we allow no rescaling. Remark that the inverse of a based isometry is again one.

In the next lemma we write
$\sigma$ also for the linear extension.
Since a based isometry preserves the Cartan integers, the equality
$\sigma(\Psi)=\Psi'$ follows from \cite[\S11.1]{Humphreys}. We include the short proof because the theorem below uses
the explicit correspondences of parts~(1) and~(3), and not simply the existence of an isomorphism.

\begin{lemma}\label{lem:basediso}
	Let $\sigma:(\Psi,\Delta)\to(\Psi',\Delta')$ be a based isometry. Then:
	\begin{enumerate}
		\item  The map $w\mapsto\sigma w\sigma^{-1}$ is an isomorphism from the Weyl group of $\Psi$
		onto that of $\Psi'$, carrying $s_\gamma$ to $s_{\sigma\gamma}$ for $\gamma\in\Delta$ and
		preserving the length function.
		
		\item $\sigma(\Psi)=\Psi'$ and $\sigma(\Psi^+)=(\Psi')^+$.
		
		\item $\Pt^{\Psi'}_q(\sigma\nu)=\Pt^{\Psi}_q(\nu)$ for every $\nu\in E$.
	\end{enumerate}
\end{lemma}

\begin{proof}
	(1) Since $\sigma$ is an isometry,
	$\sigma s_\gamma\sigma^{-1}=s_{\sigma\gamma}$ for every $\gamma\in\Delta$. Hence, if
	$w=s_{\gamma_1}\cdots s_{\gamma_r}\in W(\Psi)$, then
	$$
	\sigma w\sigma^{-1}=s_{\sigma\gamma_1}\cdots s_{\sigma\gamma_r}\in W(\Psi').
	$$
	Thus conjugation by $\sigma$ gives a homomorphism from $W(\Psi)$ to $W(\Psi')$. Applying the same argument to
	$\sigma^{-1}$ gives an inverse homomorphism, so conjugation by $\sigma$ is an isomorphism.
	
	If the expression for $w$ above is reduced, then so is the corresponding expression for
	$\sigma w\sigma^{-1}$, because otherwise a shorter expression, conjugated by $\sigma^{-1}$, would give a
	shorter expression for $w$. Thus the isomorphism preserves the length function.
	
	(2) Every root of $\Psi$ has the form $w(\gamma)$ with $w\in W(\Psi)$ and $\gamma\in\Delta$. By (1),
	$$
	\sigma(w\gamma)=(\sigma w\sigma^{-1})(\sigma\gamma)\in\Psi',
	$$
	so $\sigma(\Psi)\subseteq\Psi'$. Applying the same argument to
	$\sigma^{-1}:(\Psi',\Delta')\to(\Psi,\Delta)$ gives
	$\sigma^{-1}(\Psi')\subseteq\Psi$, thus  $\Psi'\subseteq\sigma(\Psi)$. We conclude that
	$\sigma(\Psi)=\Psi'$.

	If $\alpha\in\Psi^+$, write
	$
	\alpha=\sum_{\gamma\in\Delta} c_\gamma\gamma
	$
	with $c_\gamma\ge0$. Then
	$
	\sigma(\alpha)=\sum_{\gamma\in\Delta} c_\gamma\,\sigma(\gamma),
	$
	which is a nonnegative combination of the elements of $\Delta'$. Since
	$\sigma(\alpha)\in\Psi'$, it follows that $\sigma(\alpha)\in(\Psi')^+$.
	Applying the same argument to $\sigma^{-1}$ gives
	$\sigma(\Psi^+)=(\Psi')^+$.

	(3) By (2), $\sigma$ carries a way of writing $\nu$ as a sum of $k$ positive roots of $\Psi$
	to a way of writing $\sigma\nu$ as a sum of $k$ positive roots of $\Psi'$. Applying the same
	argument to $\sigma^{-1}$ gives the inverse correspondence, so
	$\Pt_q^{\Psi'}(\sigma\nu)=\Pt_q^\Psi(\nu)$.
\end{proof}

The statement of the theorem below refers to the \emph{restricted pair} of the realization, which
we now introduce: set
\[
\mu_I\;\coloneqq\;\sum_{i\in I}\langle\mu,\alpha_i^\vee\rangle\,\tilde\varpi_i\;\in\;U,
\qquad
\lambda_I\;\coloneqq\;\mu_I+\beta .
\]
For every $i\in I$, by definition of the $\tilde\varpi_i$ one has
$\langle\mu_I,\alpha_i^\vee\rangle=\langle\mu,\alpha_i^\vee\rangle$, and hence
$\langle\lambda_I,\alpha_i^\vee\rangle=\langle\lambda,\alpha_i^\vee\rangle$. This implies that
$\mu_I$ and $\lambda_I$ are dominant weights of $\Phi_I$, and $\mu_I\le\lambda_I$ in the dominance
order of $\Phi_I$ since \[\lambda_I-\mu_I=\beta\in\Z_{\ge0}\{\alpha_i:i\in I\}\] by \eqref{beta}.   Geometrically, $\mu_I$ and $\lambda_I$ are the orthogonal projections of
$\mu$ and $\lambda$ onto $U$: one has $(\mu_I - \mu, \alpha_i) = 0$ for
every $i \in I$, and the $\alpha_i$ with $i \in I$ span $U$. That
nothing is lost in discarding the coordinates at the nodes outside $I$
is the content of part (1) of the theorem below.

\begin{theorem}\label{thm:parab}
Let $P=P^\Phi_{\mu,\lambda}$ be a realized polytope with $I=\supp(M)\ne\varnothing$.
\begin{enumerate}
\item Formula~\eqref{eq:KP} is computed entirely inside $\Phi_I$:
\[
\Kf_\Phi(M,\delta)=K^{\Phi_I}_{\lambda_I\mu_I}.
\]
\item Let $P'=P^{\Phi'}_{\mu',\lambda'}$ be a second realized polytope, with shape data
$(M',\delta')$, support $I'=\supp(M')$ and set of cutting walls $\mathrm{Cut}'$. Suppose that there
is a bijection $\pi:I\to I'$ such that $\sigma:\alpha_i\mapsto\alpha'_{\pi i}$ is a based isometry
$(\Phi_I,\{\alpha_i\}_{i\in I})\to(\Phi'_{I'},\{\alpha'_{i'}\}_{i'\in I'})$ and
\[
m'_{\pi i}=m_i\quad(i\in I),\qquad \pi(\mathrm{Cut})=\mathrm{Cut}',\qquad
\delta'_{\pi i}=\delta_i\quad(i\in\mathrm{Cut}).
\]
Then $\Kf_\Phi(M,\delta)=\Kf_{\Phi'}(M',\delta')$.
\end{enumerate}
\end{theorem}

\begin{proof}

(1) Throughout the proof, a superscript $\Phi$ or $\Phi_I$ indicates the
realization in which a quantity is computed. Thus, for instance, we write
$\beta^\Phi$, $m_i^\Phi$, $\hat\mu_i^\Phi$, etc. for the data associated with
$(\Phi,\mu,\lambda)$, and use the corresponding notation with superscript
$\Phi_I$ for $(\Phi_I,\mu_I,\lambda_I)$.

We first compare the active sets. By Lemma~\ref{lem:JCutI},
$J^\Phi\subseteq I$. Since
$\beta^{\Phi_I}=\lambda_I-\mu_I=\beta^\Phi
=\sum_{i\in I}m_i^\Phi\alpha_i$,
we have $m_i^{\Phi_I}=m_i^\Phi$ for every $i\in I$.
Moreover, by \eqref{lambdamu},
$\hat\mu_i^{\Phi_I}=\hat\mu_i^\Phi$ and
$\hat\lambda_i^{\Phi_I}=\hat\lambda_i^\Phi$ for every $i\in I$.
By \eqref{deltai} we have $\delta_i^{\Phi_I}=\delta_i^\Phi$ for every $i\in I$. Therefore
$$
J^{\Phi_I}
=\{i\in I:\delta_i^{\Phi_I}\ge m_i^{\Phi_I}\}
=\{i\in I:\delta_i^\Phi\ge m_i^\Phi\}
=J^\Phi,
$$
where the last equality uses $J^\Phi\subseteq I$. We will write
simply $J$ for this common active set.

Write $W^\Phi$ and $W^{\Phi_I}$ for the Weyl groups of $\Phi$ and
$\Phi_I$, respectively, and $W_J^\Phi$ and $W_J^{\Phi_I}$ for their
standard parabolic subgroups generated by $\{s_j:j\in J\}$.
For $j\in J$, the reflection $s_j$ preserves $U$ and fixes $U^\perp$
pointwise. Its restriction to $U$ is the simple reflection of $\Phi_I$
corresponding to $\alpha_j$. Thus restriction to $U$ gives an isomorphism
$W_J^\Phi\xrightarrow{\sim}W_J^{\Phi_I}$, sending each simple generator
$s_j$ to the corresponding simple generator, hence it preserves the
Coxeter lengths in these two parabolic groups. By
\cite[Ch.~IV, \S1, no.~8, Prop.~7, Cor.~4]{Bourbaki}, these lengths agree
with the ambient lengths in $W^\Phi$ and $W^{\Phi_I}$, respectively.
Therefore, under this identification,
$\ell_{W^\Phi}(w)=\ell_{W^{\Phi_I}}(w)$ for every $w\in W_J^\Phi$.
In particular, the sign $(-1)^{\ell(w)}$ in \eqref{eq:KP} is unchanged
when passing from $\Phi$ to $\Phi_I$.

Fix $w\in W_J^\Phi$, and choose a reduced expression
$w=s_{j_1}\cdots s_{j_r}$, with $j_1,\dots,j_r\in J$. Set
$u_k=s_{j_1}\cdots s_{j_{k-1}}$. Since $u_k\in W_J^\Phi$,
the root $u_k(\alpha_{j_k})$ lies in
$\operatorname{span}\{\alpha_j:j\in J\}\cap\Phi=\Phi_J$.
It is also a positive root of $\Phi$ by the argument following
\eqref{eq:telescope}; hence
$u_k(\alpha_{j_k})\in\Phi_J^+\subseteq\Phi_I^+$.
Using \eqref{eq:vw} and \eqref{beta}, we deduce that
$v_w^\Phi\in\mathbb Z\{\alpha_i:i\in I\}$.

Under the identification $W_J^\Phi\simeq W_J^{\Phi_I}$, the same reduced
expression for $w$ has the same prefixes $u_k$. Since
$\beta^{\Phi_I}=\beta^\Phi$ and
$\hat\lambda_j^{\Phi_I}=\hat\lambda_j^\Phi$ for every $j\in J$,
\eqref{eq:vw} gives
$$
v_w^{\Phi_I}
=\beta^{\Phi_I}
-\sum_{k=1}^r\hat\lambda_{j_k}^{\Phi_I}u_k(\alpha_{j_k})
=\beta^\Phi
-\sum_{k=1}^r\hat\lambda_{j_k}^\Phi u_k(\alpha_{j_k})
=v_w^\Phi.
$$

We next compare the two partition functions on such vectors. Let
$\nu=\sum_{i\in I}b_i\alpha_i$, with $b_i\in\mathbb Z$.
By Remark~\ref{elrem}, applied to $\Phi$ and to $\Phi_I$, both
$\Pt_q^\Phi(\nu)$ and $\Pt_q^{\Phi_I}(\nu)$ vanish if some $b_i<0$.
Suppose therefore that all $b_i\ge0$, and consider a Kostant partition
$\nu=\sum_{\gamma\in\Phi^+}n_\gamma\gamma$, with
$n_\gamma\in\mathbb Z_{\ge0}$. Write
$\gamma=\sum_k c_{\gamma,k}\alpha_k$, with $c_{\gamma,k}\ge0$.
For $k\notin I$, the coefficient of $\alpha_k$ in the expansion of $\nu$
gives $0=\sum_\gamma n_\gamma c_{\gamma,k}$. Since every summand is
nonnegative, whenever $n_\gamma>0$ we have $c_{\gamma,k}=0$ for every
$k\notin I$. Thus every root occurring in the partition lies in
$\Phi\cap U=\Phi_I$, and hence in $\Phi_I^+$. Therefore the Kostant
partitions of $\nu$ in $\Phi$ and in $\Phi_I$ are the same, with the same
number of parts. It follows that
$\Pt_q^\Phi(\nu)=\Pt_q^{\Phi_I}(\nu)$.

Combining the preceding identifications and equalities gives
$\Kf_\Phi(M,\delta)=K^{\Phi_I}_{\lambda_I\mu_I}$.

(2) Write $\sigma$ also for the linear extension $U\to U'$. By
Lemma~\ref{lem:basediso}, conjugation by $\sigma$ is a length-preserving
isomorphism of Weyl groups matching $s_j$ with $s'_{\pi j}$, and
$$\Pt^{\Phi'_{I'}}_q(\sigma\nu)=\Pt^{\Phi_I}_q(\nu).$$	We claim that $\pi(J)=J'$. Since $J\subseteq\mathrm{Cut}$,
$J'\subseteq\mathrm{Cut}'$, and $\pi(\mathrm{Cut})=\mathrm{Cut}'$,
it is enough to consider $i\in\mathrm{Cut}$. For such $i$,
$m'_{\pi i}=m_i$ and $\delta'_{\pi i}=\delta_i$, hence
$$
i\in J
\iff \delta_i\ge m_i
\iff \delta'_{\pi i}\ge m'_{\pi i}
\iff \pi i\in J'.
$$
Therefore $\pi(J)=J'$.   Hence
conjugation by $\sigma$ restricts to a length-preserving isomorphism $W_J\to W'_{J'}$. Moreover, by \eqref{deltai},  for every $j\in J$,
\[
\hat\lambda'_{\pi j}
=2m'_{\pi j}-\delta'_{\pi j}
=2m_j-\delta_j
=\hat\lambda_j,
\]
and, by \eqref{beta},
\[
\beta'
=\sum_{i\in I}m'_{\pi i}\alpha'_{\pi i}
=\sigma(\beta).
\]

Let $w\in W_J$ have a reduced expression $s_{j_1}\cdots s_{j_r}$, and put
$w'=\sigma w\sigma^{-1}$. Then $w'\in W'_{J'}$ has a reduced expression
$s'_{\pi j_1}\cdots s'_{\pi j_r}$, with prefixes
$u'_k=\sigma u_k\sigma^{-1}$. Then by \eqref{eq:vw}
\[
\begin{aligned}
	v'_{w'}
	&=\beta'-\sum_{k=1}^r\hat\lambda'_{\pi j_k}\,
	u'_k(\alpha'_{\pi j_k})\\
	&=\sigma(\beta)-\sum_{k=1}^r\hat\lambda_{j_k}\,
	\sigma\bigl(u_k(\alpha_{j_k})\bigr)\\
	&=\sigma(v_w).
\end{aligned}
\] 
By part~(1), both sides may be computed in their respective support root systems. Applying \eqref{eq:KP} inside $\Phi'_{I'}$ and $\Phi_I$, we obtain
\[
\begin{aligned}
	\Kf_{\Phi'}(M',\delta')
	&=\sum_{w'\in W'_{J'}}
	(-1)^{\ell(w')}\,
	\Pt^{\Phi'_{I'}}_q(v'_{w'})\\
	&=\sum_{w\in W_J}
	(-1)^{\ell(w)}\,
	\Pt^{\Phi'_{I'}}_q\bigl(\sigma(v_w)\bigr)\\
	&=\sum_{w\in W_J}
	(-1)^{\ell(w)}\,
	\Pt^{\Phi_I}_q(v_w)\\
	&=\Kf_{\Phi}(M,\delta).
\end{aligned}
\]

\end{proof}

\section{Recovery of the root system from the shape}\label{sec:recovery}

We now shift towards the geometry: Proposition~\ref{prop:recover} recovers, from the raw shape alone, the
based root system $\Phi_I$ together with the widths, the cutting walls and their depths, so that
Theorem~\ref{thm:parab}(2) applies (Figure~\ref{fig:recovery} summarizes how this argument goes). This is the key missing ingredient for the
proof of Theorem~\ref{thm:main}.

Let $P=P^\Phi_{\mu,\lambda}\subseteq \mu +U.$ 
We will see in the first paragraph of the proof of Lemma~\ref{lem:inventory} below (which does not use the following notion) that $P$ is full-dimensional in $\mu+U$; hence the following is well defined.
By the \emph{outer normal direction} of a facet $F$ of $P$ we mean a nonzero vector $u\in U$, well
defined up to a positive scalar, such that the linear functional $(\,\cdot\,,u)$ attains its maximum
over $P$ exactly at the points of $F$. Note that this is a genuine direction and not merely a line:
$-u$ is not an outer normal direction of $F$, since $(\,\cdot\,,-u)$ attains its minimum
there.

Every
facet of a polyhedron lies on the boundary hyperplane of one of its defining inequalities, whose outer
normal is then that of the facet.
The facet normals of $P$ can be read directly from the three sets of defining inequalities
\begin{equation}\label{Peqs}
P=\{v\in \mu +U\, \vert\, a_i(v)\ge 0, a_i(v)\le m_i\ \mathrm{for}\ i\in I \ \mathrm{and}\ \tau_i(v)\ge 0\ \mathrm{for} \ i\in \mathrm{Cut}\}.
\end{equation}
 By Equation~\eqref{ai},
the hyperplanes $a_i=0$ and $a_i=m_i$ are orthogonal to $\widetilde{\varpi}_i$. Since $P$ lies between them, their outer normals are $-\widetilde{\varpi}_i$ and $\widetilde{\varpi}_i$, respectively. By \eqref{tau}, we have
\begin{equation}\label{sigmai}
\tau_i(v)=\langle v,\alpha_i^\vee\rangle+1
=\frac{(v,\alpha_i)}{D_i}+1,
\end{equation}

so the hyperplane $\tau_i=0$ is orthogonal to $\alpha_i$. Since $P$ lies in the half-space $\tau_i\ge 0$, its outer normal is $-\alpha_i$. Therefore, every facet normal of $P$ is among
$
-\widetilde{\varpi}_i,
\widetilde{\varpi}_i$ for 
 $i\in I$, 
and 
 $-\alpha_i$
 for $ i\in\mathrm{Cut}.$

\begin{lemma}\label{lem:dichotomy}
Among the directions $\{\pm\tilde\varpi_i\}_{i\in I}\cup\{-\alpha_j\}_{j\in\mathrm{Cut}}$, the only
pairs spanning the same line are the antipodal parallelepiped pairs $\{+\tilde\varpi_i,-\tilde\varpi_i\}$.
\end{lemma}

\begin{proof}
The $\tilde\varpi_i$ are a basis of $U$, thus pairwise non-parallel. Likewise, the $\alpha_i$ ($i\in I$) form a basis of $U$, hence are pairwise non-parallel. Supposing $\alpha_i\parallel
\tilde\varpi_j$ with $i\ne j$, gives a contradiction because $(\alpha_i,\tilde\varpi_j)=D_i\delta_{ji}=0$. Assume finally that $i\in\mathrm{Cut}$ and $\alpha_i\parallel\tilde\varpi_i$, and write $\tilde\varpi_i=c\,\alpha_i$. Pairing
with $\alpha_k^\vee$ gives $c\langle\alpha_i,\alpha_k^\vee\rangle=\delta_{ik}$ for all $k\in I$, so
$c=\tfrac12\neq 0$.  By \eqref{beta} one has $\langle\beta,\alpha_i^\vee\rangle=2m_i$, so by \eqref{deltai} and \eqref{lambdamu}
$\delta_i=-\hat\mu_i\le-1$, which contradicts  $i\in\mathrm{Cut}$.
\end{proof}
 
 Up to translation by $\rho$, our polytope $P$ is obtained from
 the dominant weight polytope associated with $\lambda+\rho$
 by intersecting it with $\mu+\rho+C$.
 The next lemma describes its facets, in the spirit of
 \cite[Theorem~2.8]{BGH}, which describes all faces of
 strongly dominant weight polytopes (a stronger result for a slightly simpler polytope).
 
\begin{lemma}\label{lem:inventory}
The polytope $P$ is full-dimensional in $\mu+U$, a space of dimension $|I|$. Its
facets are exactly the following, each occurring once:
\begin{itemize}
\item for each $i\in I$, the two \emph{parallelepiped facets} $P\cap\{a_i=0\}$ and $P\cap\{a_i=m_i\}$, with
outer normals $-\tilde\varpi_i$ and $+\tilde\varpi_i$ respectively;
\item for each cutting wall $i\in\mathrm{Cut}$, the \emph{wall facet} $P\cap\{\tau_i=0\}$, with outer
normal $-\alpha_i$.
\end{itemize}
In particular, $P$ has exactly $2|I|+|\mathrm{Cut}|$ facets.
\end{lemma}

\begin{proof}

We first prove that $P$ is full-dimensional. At $v=\mu$, all wall inequalities are strict, since $\tau_i(\mu)=\widehat{\mu}_i\ge 1$. Choose $\varepsilon>0$ sufficiently small and set
\[
v_\varepsilon=\mu+\varepsilon\sum_{i\in I}\alpha_i.
\]
Then $0<a_i(v_\varepsilon)=\varepsilon<m_i$ for every $i\in I$, and $\tau_i(v_\varepsilon)>0$ for every $i\in\mathrm{Cut}$. Thus all the inequalities defining $P$ are strict at $v_\varepsilon$, so a neighborhood of $v_\varepsilon$ in $\mu+U$ is contained in $P$. Then $\dim P=\dim U=|I|$.

Each boundary hyperplane defining $P$ has codimension one in
$\mu+U$: this is immediate from the expressions \eqref{ai} for $a_i$ and \eqref{sigmai} for
$\tau_i$.

To show that one of these inequalities defines a facet, it is enough
to find a point  on its boundary hyperplane at which all the other
defining inequalities are strict. By continuity, the corresponding
face then contains a relatively open neighborhood of that point in
the boundary hyperplane, and therefore has dimension $|I|-1$.

Fix $i\in I$. As explained above, at $v=\mu$ all wall inequalities are strict. Keeping $a_i=0$, set $a_l=\varepsilon$ for every $l\in I\smallsetminus\{i\}$. For $\varepsilon>0$ sufficiently small, all the other parallelepiped inequalities and all other wall inequalities remain strict. Hence $P\cap \{a_i=0\}$ is a facet, with outer normal $-\widetilde{\varpi}_i$.

The same argument applies to $P\cap\{a_i=m_i\}$, for $i\in I$ fixed. Starting from $v=\lambda$, keep $a_i=m_i$ and set $a_l=m_l-\varepsilon$ for every $l\in I\smallsetminus\{i\}$. Since all wall inequalities are strict at $\lambda$, they remain strict for sufficiently small $\varepsilon>0$. Thus $P\cap\{a_i=m_i\}$ is a facet, with outer normal $\widetilde{\varpi}_i$.

We now treat the case of wall facets.
Fix $i\in\mathrm{Cut}$. We first construct a point $y_0$ on the
$i$-th wall which strictly satisfies all the other wall inequalities.
This point may lie on some upper facets of the parallelepiped, so we
then move it slightly, along the $i$-th wall, into the interior of
$\Pi$. The resulting point $y_\varepsilon$, for sufficiently small
$\varepsilon>0$, will be the desired point.

Set
\[
y_0
=
\mu
+
\frac{\delta_i}{2}\alpha_i
+
\sum_{k\in I\smallsetminus\{i\}}m_k\alpha_k.
\]
Since $1\le\delta_i\le 2m_i-1$, the coefficient
$\delta_i/2$ lies strictly between $0$ and $m_i$. Moreover, \linebreak
$y_0=\lambda-(m_i-\delta_i/2)\alpha_i$, and hence by \eqref{tau}
\[
\tau_i(y_0)
=
\widehat{\lambda}_i
-
2\left(m_i-\frac{\delta_i}{2}\right)
=
0.
\]      
For $k\neq i$, using \eqref{lambdamu} we have
\[
\tau_k(y_0)
=
\widehat{\lambda}_k
-
\left(m_i-\frac{\delta_i}{2}\right)
\langle\alpha_i,\alpha_k^\vee\rangle
\ge
\widehat{\lambda}_k
\ge 1,
\]
because $\langle\alpha_i,\alpha_k^\vee\rangle\le 0$.   Thus $y_0$
lies on the $i$-th wall and strictly satisfies every other wall
inequality. However, $a_k(y_0)=m_k$ for every $k\neq i$, and so we need to slightly perturb this vector.

For $\varepsilon>0$, define
\[
y_\varepsilon
=
y_0
-
\varepsilon
\sum_{k\in I\smallsetminus\{i\}}
\left(
\alpha_k
+
\frac{1}{2}
\left|
\langle\alpha_k,\alpha_i^\vee\rangle
\right|
\alpha_i
\right).
\]
Each vector in the sum is parallel to the $i$-th wall, since
\[
\left\langle
\alpha_k
+
\frac{1}{2}
\left|
\langle\alpha_k,\alpha_i^\vee\rangle
\right|
\alpha_i,
\alpha_i^\vee
\right\rangle
=
\langle\alpha_k,\alpha_i^\vee\rangle
+
\left|
\langle\alpha_k,\alpha_i^\vee\rangle
\right|
=
0.
\] 
Thus $\tau_i(y_\varepsilon)=0$ for every $\varepsilon$.
For sufficiently small $\varepsilon>0$, we have
$0<a_k(y_\varepsilon)<m_k$ for all $k\in I$ and, by continuity,
$\tau_k(y_\varepsilon)>0$ for all
$k\in\mathrm{Cut}\smallsetminus\{i\}$.
Hence $P\cap\{\tau_i=0\}$ is a facet with outer normal $-\alpha_i$.

Finally, the $2|I|+|\mathrm{Cut}|$ facets just produced are pairwise different. Two of them could
coincide only if their boundary hyperplanes did, hence only if their outer normals spanned the same
line; by Lemma~\ref{lem:dichotomy} the only such pairs are $\{\pm\tilde\varpi_i\}$, whose
hyperplanes $\{a_i=0\}$ and $\{a_i=m_i\}$ are distinct because $m_i>0$.
\end{proof}

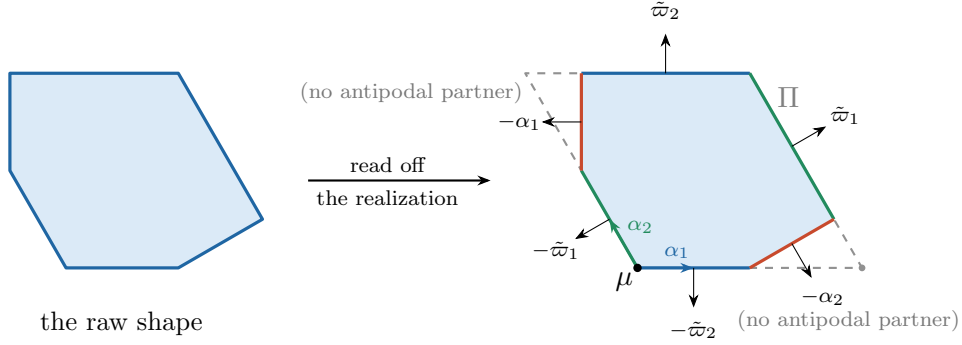
\begin{figure}[ht]
\centering
\begin{tikzpicture}[scale=1.05,>={Stealth[length=5pt]},line join=round]
 \begin{scope}
  \coordinate (h1) at (0,0);
  \coordinate (h2) at (1.4142,0);
  \coordinate (h3) at (2.4749,0.6124);
  \coordinate (h4) at (1.4142,2.4495);
  \coordinate (h5) at (-0.7071,2.4495);
  \coordinate (h6) at (-0.7071,1.2247);
  \fill[polyfill] (h1)--(h2)--(h3)--(h4)--(h5)--(h6)--cycle;
  \draw[boxblue,very thick] (h1)--(h2)--(h3)--(h4)--(h5)--(h6)--cycle;
  \node[font=\small] at (0.7,-0.7) {the raw shape};
 \end{scope}
 \draw[->,thick] (3.05,1.1)--(5.35,1.1);
 \node[font=\scriptsize,align=center] at (4.05,1.3) {read off};
 \node[font=\scriptsize,align=center] at (4.05,0.88) {the realization};
 \begin{scope}[xshift=7.2cm]
  \coordinate (O)   at (0,0);
  \coordinate (P1)  at (2.8284,0);
  \coordinate (P12) at (1.4142,2.4495);
  \coordinate (P2)  at (-1.4142,2.4495);
  \coordinate (h1) at (0,0);
  \coordinate (h2) at (1.4142,0);
  \coordinate (h3) at (2.4749,0.6124);
  \coordinate (h4) at (1.4142,2.4495);
  \coordinate (h5) at (-0.7071,2.4495);
  \coordinate (h6) at (-0.7071,1.2247);
  \draw[ghost!130,dashed,thick] (O)--(P1)--(P12)--(P2)--cycle;
  \fill[ghost!130] (P1) circle (1.2pt);
  \node[ghost!150] at (1.9,2.15) {$\Pi$};
  \fill[polyfill] (h1)--(h2)--(h3)--(h4)--(h5)--(h6)--cycle;
  \draw[boxblue,very thick] (h1)--(h2);
  \draw[boxblue,very thick] (h4)--(h5);
  \draw[boxgreen,very thick] (h3)--(h4);
  \draw[boxgreen,very thick] (h6)--(h1);
  \draw[clipred,very thick] (h2)--(h3);
  \draw[clipred,very thick] (h5)--(h6);
  \draw[->,black] (0.7071,0)--(0.7071,-0.5);
  \node[black,font=\scriptsize] at (0.7071,-0.78) {$-\tilde\varpi_2$};
  \draw[->,black] (0.3536,2.4495)--(0.3536,2.9495);
  \node[black,font=\scriptsize] at (0.3536,3.24) {$\tilde\varpi_2$};
  \draw[->,black] (1.9446,1.531)--(2.3776,1.781);
  \node[black,font=\scriptsize] at (2.6201,1.921) {$\tilde\varpi_1$};
  \draw[->,black] (-0.3536,0.6124)--(-0.7866,0.3624);
  \node[black,font=\scriptsize] at (-1.0291,0.2224) {$-\tilde\varpi_1$};
  \draw[->,black] (1.9446,0.3062)--(2.1946,-0.1268);
  \node[black,font=\scriptsize] at (2.3346,-0.3693) {$-\alpha_2$};
  \node[ghost!160,font=\scriptsize] at (2.66,-0.66) {(no antipodal partner)};
  \draw[->,black] (-0.7071,1.8371)--(-1.2071,1.8371);
  \node[black,font=\scriptsize] at (-1.4871,1.8371) {$-\alpha_1$};
  \node[ghost!160,font=\scriptsize] at (-2.85,2.20) {(no antipodal partner)};
  \draw[->,boxblue,thick] (O)--(0.7071,0);
  \draw[->,boxgreen,thick] (O)--(-0.3536,0.6124);
  \node[boxblue,font=\scriptsize,below=0pt] at (0.5,0.4) {$\alpha_1$};
  \node[boxgreen,font=\scriptsize,right=-1pt] at (-0.22,0.52) {$\alpha_2$};
  \fill[black] (O) circle (1.5pt);
  \node[below left=-3pt] at (O) {$\mu$};
 \end{scope}
\end{tikzpicture}
\caption{\textbf{Recovering the realization from the raw shape}: From the raw shape (left) one reads off its
facet normals (right). Four of them form two antipodal pairs: these are the normals of the
parallelepiped $\Pi$, and extending the corresponding facets rebuilds $\Pi$. The other two,
$-\alpha_1$ and $-\alpha_2$, have no antipodal partner: they are the normals of the cut walls. The
edges of $\Pi$ leaving $\mu$ are $m_1\alpha_1$ and $m_2\alpha_2$; their lengths give the widths
$m_i$ and the root lengths $\lVert\alpha_i\rVert$, and the angle between them fixes the Cartan
integers, hence the Dynkin diagram.}
\label{fig:recovery}
\end{figure}

Let $P=P^\Phi_{\mu,\lambda}$ and $P'=P^{\Phi'}_{\mu',\lambda'}$
be two realized polytopes, with support spaces $U$ and $U'$, respectively.
Recall that the \emph{affine hull} of a subset $B$ of a Euclidean space, denoted $\operatorname{aff}(B)$, is the
smallest affine subspace containing it.
By Lemma~\ref{lem:inventory}, $\operatorname{aff}(P)=\mu+U$ and
$\operatorname{aff}(P')=\mu'+U'$.   A surjective isometry between two subsets of Euclidean spaces extends uniquely to an
affine isometry of their affine hulls. 
 Hence every isometry\footnote{Here isometry means a distance-preserving bijection for the induced Euclidean
	metrics, with both root systems normalized so that long roots have length $\sqrt2$ (as in Theorem~\ref{thm:main}); in
	particular, no rescaling is allowed. The ambient spaces $V$ and $V'$ may have
	different dimensions.} $\varphi:P\to P'$ extends uniquely
to an affine isometry
\[
\widehat\varphi:\mu+U\longrightarrow\mu'+U'.
\]

For an affine map $f$ between affine subspaces we write $L_f$ for its \emph{linear part}, the linear
map characterized by $L_f(v-w)=f(v)-f(w)$; we abbreviate $L_\varphi\coloneqq L_{\widehat\varphi}:U\to U'$.
Thus $\widehat\varphi$ is a translation composed with its linear part,
\[
  \widehat\varphi(v)=\varphi(\mu)+L_\varphi(v-\mu)\qquad(v\in\mu+U),
\]
and $L_\varphi$, being a linear isometry, preserves inner products, lengths and angles. Linear parts
behave well under the operations we shall perform on isometries: $L_{f\circ h}=L_f\circ L_h$ and
$L_{f^{-1}}=L_f^{-1}$. It is through $L_\varphi$ that an isometry of the raw shapes transports the
root-theoretic data of one realization to the other.

\begin{lemma}\label{lem:normals}
Let $\varphi:P\to P'$ be an isometry of realized polytopes. If $F$ is a facet of $P$ with outer
normal direction $u$, then $\widehat\varphi(F)$ is a facet of $P'$ with outer normal direction
$L_\varphi(u)$. Then $L_\varphi$ carries the set of outer normal directions of the facets  of $P$
bijectively onto that of $P'$, sending antipodal pairs to antipodal pairs and directions without
antipodal partner to directions without antipodal partner.
\end{lemma}

\begin{proof}
The affine bijection $\widehat\varphi$ carries $P$ onto $P'$, hence faces to faces and facets to
facets. For $v\in\mu+U$,
\[
\bigl(\widehat\varphi(v),\,L_\varphi(u)\bigr)
=\bigl(\varphi(\mu)+L_\varphi(v-\mu),\,L_\varphi(u)\bigr)
=(v,u)+c,
\qquad c\coloneqq \bigl(\varphi(\mu),L_\varphi(u)\bigr)-(\mu,u),
\]
because $L_\varphi$ preserves inner products. Hence $(\,\cdot\,,L_\varphi(u))$ attains its maximum
on $P'$ exactly on $\widehat\varphi(F)$, which is the assertion. The last sentence follows because
$L_\varphi$ is a linear map. 
\end{proof}

The matching of edge generators below is similar to Gover's
description of congruence for zonotopes \cite[Theorem~2.2]{Gover}.
Our parallelepipeds are particular cases of zonotopes (called
parallelotopes in his paper).

\begin{proposition}\label{prop:recover}
Let $P=P^{\Phi}_{\mu,\lambda}$ and $P'=P^{\Phi'}_{\mu',\lambda'}$ be realized polytopes, and let
$\varphi:P\to P'$ be an isometry, with affine extension $\widehat\varphi$ and linear part
$L_\varphi:U\to U'$ as above. Then:
\begin{enumerate}
\item $\widehat\varphi(\Pi)=\Pi'$.
\item\label{perm} There is a bijection $\pi:I\to I'$ such that $  L_\varphi(\alpha_i)=\varepsilon_i\,\alpha'_{\pi i} $ for $i\in I$, for certain signs $\varepsilon_i\in\{+1,-1\}$. Moreover, after composing $\varphi$ with a self-isometry $\iota$ of $P'$,
the signs can all be chosen to be $+1$:
$L_{\iota\circ\varphi}(\alpha_i)=\alpha'_{\pi i}$ for $i\in I$.
In particular, $\sigma:\alpha_i\mapsto\alpha'_{\pi i}$ is a based
isometry
$(\Phi_I,\{\alpha_i\}_{i\in I})\to
(\Phi'_{I'},\{\alpha'_{i'}\}_{i'\in I'})$.
\item The same bijection $\sigma$ matches the rest of the shape data: $m'_{\pi i}=m_i$ for every $i\in I$, and
$\pi(\mathrm{Cut})=\mathrm{Cut}'$ with $\delta'_{\pi i}=\delta_i$ for every $i\in\mathrm{Cut}$.
\end{enumerate}
\end{proposition}

\begin{proof}

We begin with (1). By Lemma~\ref{lem:normals} the map $L_\varphi$ carries the outer facet-normal
directions of $P$ bijectively onto those of $P'$, preserving antipodality. 	By Lemmas~\ref{lem:dichotomy} and~\ref{lem:inventory}, applied to $P$ and to $P'$, the only
antipodal pairs of facet-normal directions are the parallelepiped pairs
$\{\pm\tilde\varpi_i\}$, while the remaining directions are precisely
the wall normals. This implies that $\widehat\varphi$ induces a bijection between the parallelepiped facets: 
\[
\mathcal F_{\mathrm{par}}(P)\longrightarrow\mathcal F_{\mathrm{par}}(P'),
\qquad
F\longmapsto\widehat\varphi(F).
\]
 For such a facet $F$, let $H_P(F)$ be the closed half-space of $\operatorname{aff}(P)$ bounded by $\operatorname{aff}(F)$ and containing $P$. By Lemma~\ref{lem:inventory},
\[
\Pi=\bigcap_{F\in\mathcal F_{\mathrm{par}}(P)}H_P(F),
\]
and analogously for $\Pi'$. Since $\widehat\varphi$ is an affine bijection carrying $P$ onto $P'$, it sends $H_P(F)$ onto $H_{P'}(\widehat\varphi(F))$. Taking intersections therefore gives
\[
\widehat\varphi(\Pi)=\Pi',
\]
proving (1). Thus $\Pi$ is canonically determined by the raw shape.

We turn to the construction of the bijection $\pi$. By part (1) the affine bijection
$\widehat\varphi$ carries $\Pi$ onto $\Pi'$, hence carries the face lattice of $\Pi$ onto that of
$\Pi'$; in particular it maps edges of $\Pi$ to edges of $\Pi'$. The edges of $\Pi$ fall into $|I|$ parallel classes, which we index by $I$
so that the $i$-th class consists of the edges parallel to $\alpha_i$.
Every edge in this class has length $m_i\|\alpha_i\|$. Since $\widehat\varphi$ preserves parallelism and lengths, it induces a bijection $\pi:I\to I'$ such that
\begin{equation}\label{dos}
L_\varphi(\R\alpha_i)=\R\alpha'_{\pi i}
\qquad\text{and}\qquad
m_i\lVert\alpha_i\rVert
=m'_{\pi i}\lVert\alpha'_{\pi i}\rVert.
\end{equation}
As the long roots have length  $\sqrt{2}$, the possible lengths of simple roots are
$\lVert\alpha\rVert\in\{\sqrt2,\,1,\,\sqrt{2/3}\}$: long roots, short roots in types $B,C,F_4$, and
short roots in $G_2$, respectively. The ratio of any two distinct values in this list is irrational, but the second  equation
of \eqref{dos} forces $\lVert\alpha_i\rVert/\lVert\alpha'_{\pi i}\rVert=m'_{\pi i}/m_i$, a ratio of positive
integers. This implies that $\lVert\alpha_i\rVert=\lVert\alpha'_{\pi i}\rVert$, then $D'_{\pi i}=D_i$, and
$m'_{\pi i}=m_i$, which is the first assertion of (3). By \eqref{dos}, the vector
$L_\varphi(\alpha_i)$ lies on the line $\R\alpha'_{\pi i}$, so
$L_\varphi(\alpha_i)=t\,\alpha'_{\pi i}$ for some real number $t$. As $L_\varphi$ is an isometry,
$\lVert\alpha_i\rVert= |t|\,\lVert\alpha'_{\pi i}\rVert$, and we have
just seen that $\lVert\alpha_i\rVert=\lVert\alpha'_{\pi i}\rVert$, so $|t|=1$. In other words
$L_\varphi(\alpha_i)=\varepsilon_i\,\alpha'_{\pi i}$
for a sign $\varepsilon_i\in\{+1,-1\}.$
This is the first assertion of (2).

Let us prove that $\pi(\mathrm{Cut})=\mathrm{Cut}'$.
For $i\in\mathrm{Cut}$, by Lemmas~\ref{lem:dichotomy} and \ref{lem:inventory}, the normal $-\alpha_i$ has no antipodal partner.
By Lemma~\ref{lem:normals},  its image
$L_\varphi(-\alpha_i)$ is therefore a positive multiple of $-\alpha'_j$
for some $j\in\mathrm{Cut}'$.
But $L_\varphi(-\alpha_i)=-\varepsilon_i\alpha'_{\pi i}$,
and distinct simple roots are not parallel, so $j=\pi i$.
The positivity of the multiple also gives $\varepsilon_i=+1$.
Thus $\pi(\mathrm{Cut})\subseteq\mathrm{Cut}'$; applying the same
argument to $\varphi^{-1}$ gives the equality.

The signs are constant on the components of $I$. Indeed, since $L_\varphi$ preserves inner products,
\[
\,(\alpha_i,\alpha_j)=\varepsilon_i\varepsilon_j(\alpha'_{\pi i},\alpha'_{\pi j})
\qquad(i,j\in I).
\]
Call $i\ne j$ in $I$ \emph{adjacent} if $(\alpha_i,\alpha_j)\ne0$, i.e.\ if they are joined in the
Dynkin diagram. The displayed identity shows that $\pi$ preserves adjacency and non-adjacency, so
$\pi$ maps the connected components of $I$ onto those of $I'$. Since $(\alpha_i,\alpha_j)\le 0$ for $i\ne j$, 
 for adjacent $i,j$ both sides of the identity are
negative, forcing $\varepsilon_i\varepsilon_j=+1$. Hence $\varepsilon$ is constant on every component
$I_c$ of $I$, and by the previous paragraph $\varepsilon\equiv+1$ on every component containing a
cutting wall.

Suppose then that $I_c$ is a component containing no cutting wall
(which we call a \emph{cut-free} component) on which $\varepsilon\equiv-1$.
Since $\pi$ matches cutting walls, $\pi(I_c)$ is also cut-free.
Consider the affine involution $\iota_c$ of $\mu'+U'$ defined by
$$
a'_i\longmapsto m'_i-a'_i \quad (i\in\pi(I_c)),
\qquad
a'_i\longmapsto a'_i \quad (i\in I'\smallsetminus\pi(I_c)).
$$
Its linear part is $-\mathrm{id}$ on $U'_{\pi(I_c)}$ and the identity
on the other component spaces. Since these spaces are mutually
orthogonal, $\iota_c$ is an affine isometry. It also preserves $\Pi'$,
since it preserves each coordinate range $[0,m'_i]$.

To see that $\iota_c$ preserves $P'$, by \eqref{Peqs} it remains to
check the cutting-wall inequalities. Every cutting wall $k$ of $P'$
lies outside $\pi(I_c)$, so orthogonality of the components gives
$\langle\alpha'_i,\alpha'^{\vee}_k\rangle=0$ for $i\in\pi(I_c)$.
Thus its signed margin \eqref{tau}
$$
\tau'_k(v)=\sum_{i\in I'}
\langle\alpha'_i,\alpha'^{\vee}_k\rangle\,a'_i+\hat\mu'_k
$$
does not depend on the coordinates $a'_i$ with $i\in\pi(I_c)$,
which are the only coordinates that $\iota_c$ changes.
Hence $\tau'_k\circ\iota_c=\tau'_k$, and $\iota_c$ restricts to
a self-isometry of $P'$.   
Replacing $\varphi$ by $\iota_c\circ\varphi$ changes the signs
$\varepsilon_i$ from $-1$ to $+1$ on $I_c$ and leaves them unchanged
on the other components.  This preserves $\pi$ and the widths. It also preserves the
matching of cutting walls, since
$\tau'_k\circ\iota_c=\tau'_k$ for every $k\in\mathrm{Cut}'$,
so $\iota_c$ fixes each cutting-wall facet setwise.  Repeating this for every component
on which $\varepsilon\equiv-1$ gives $\varepsilon\equiv+1$ on all
of $I$. The composite $\iota$ of these involutions is the
self-isometry asserted in (2).
At this point $L_{\iota\circ\varphi}(\alpha_i)=\alpha'_{\pi i}$ for every $i\in I$, and since
$L_{\iota\circ\varphi}$ is an isometry the Gram matrices coincide:
$(\alpha'_{\pi i},\alpha'_{\pi j})=(\alpha_i,\alpha_j)$ for all $i,j\in I$. By
Definition~\ref{def:basediso}, $\sigma:\alpha_i\mapsto\alpha'_{\pi i}$ is then a based isometry
$\Phi_I\to\Phi'_{I'}$, which proves (2).

	It only remains to match the depths of the cutting walls. There are easier ways to prove this result, but we will need this geometric calculation in Theorem \ref{prop:KPshape} to produce a polynomial from the polytope. Fix $i\in\mathrm{Cut}$. The wall facet of $P$
lies on the hyperplane $\{\tau_i=0\}=\{\langle v,\alpha_i^\vee\rangle=-1\}$, and 
$\{\tau_i=-\delta_i\}=\{\langle v,\alpha_i^\vee\rangle=-1-\delta_i\}$ is a supporting hyperplane
of $\Pi$, by the definition of $\delta_i$.  Writing $\langle v,\alpha_i^\vee\rangle=(v,\alpha_i)/D_i$,
these hyperplanes are given by $(v,\alpha_i)=-D_i$ and $(v,\alpha_i)=-(1+\delta_i)D_i$. Since $\alpha_i\in U$, the Euclidean distance in $\mu+U$ between
the parallel hyperplanes $(v,\alpha_i)=r$ and $(v,\alpha_i)=r'$
is $|r-r'|/\lVert\alpha_i\rVert$. So the distance between our two
hyperplanes is
\[
\frac{\delta_i\,D_i}{\lVert\alpha_i\rVert}
\;=\;\frac{\delta_i\,D_i}{\sqrt{2D_i}}
\;=\;\delta_i\sqrt{D_i/2}\,.
\]

Since $\widehat\varphi(P)=P'$ and $\widehat\varphi$ carries the $i$-th wall facet $F_i$ of $P$ onto
the $\pi i$-th wall facet of $P'$, it carries $\operatorname{aff}(F_i)=\{\tau_i=0\}$ onto
$\{\tau'_{\pi i}=0\}$, and the closed half-space of $\mu+U$ bounded by $\{\tau_i=0\}$ and
containing $P$ onto its counterpart, that is, $\{\tau_i\ge0\}$ onto $\{\tau'_{\pi i}\ge0\}$.   
Now of the two supporting hyperplanes of $\Pi$ parallel to $\{\tau_i=0\}$, exactly one lies in
the open half-space $\{\tau_i<0\}$, namely $\{\tau_i=-\delta_i\}$: indeed
$\min_{\Pi}\tau_i=-\delta_i<0$, while $\max_{\Pi}\tau_i>0$ because $P\subseteq\{\tau_i\ge0\}$
is full-dimensional. The same holds for $\Pi'$; since $\widehat\varphi(\Pi)=\Pi'$, it follows that
$\widehat\varphi$ carries $\{\tau_i=-\delta_i\}$ onto $\{\tau'_{\pi i}=-\delta'_{\pi i}\}$. As
$\widehat\varphi$ preserves distances, we obtain
\[
\delta_i\sqrt{\frac{D_i}{2}}
=\delta'_{\pi i}\sqrt{\frac{D'_{\pi i}}{2}}.
\]
Since $D'_{\pi i}=\|\alpha'_{\pi i}\|^2/2
=\|\alpha_i\|^2/2=D_i$, it follows that
$\delta'_{\pi i}=\delta_i$, completing (3).
 
\end{proof}

\section{Polytopal invariance and the two dreams}

In this section we first prove the main theorem and then solve the dreams (D2) and (D1) from the introduction.

\begin{proof}[Proof of Theorem~\ref{thm:main}]
	Let $P=P^{\Phi}_{\mu,\lambda}$ and $P'=P^{\Phi'}_{\mu',\lambda'}$ be isometric. The case of empty support was settled at the beginning of
	Section~\ref{sec:parab}.  Assume then that $I\ne\varnothing\ne I'$, and let
	$\varphi:P\to P'$ be an isometry. By Proposition~\ref{prop:recover} there are a bijection
	$\pi:I\to I'$ and a based isometry $\sigma:\alpha_i\mapsto\alpha'_{\pi i}$ with
	\[
	m'_{\pi i}=m_i\quad(i\in I),\qquad \pi(\mathrm{Cut})=\mathrm{Cut}',\qquad
	\delta'_{\pi i}=\delta_i\quad(i\in\mathrm{Cut}).
	\]
	These are exactly the hypotheses of Theorem~\ref{thm:parab}(2), which gives
	$\Kf_{\Phi}(M,\delta)=\Kf_{\Phi'}(M',\delta')$. By Proposition~\ref{prop:KP}, the two sides equal
	$K^{\Phi}_{\lambda\mu}$ and $K^{\Phi'}_{\lambda'\mu'}$, respectively.
\end{proof}

\begin{remark}
	Theorem~\ref{thm:main} remains valid for reducible root systems
	$\Phi$ and $\Phi'$, provided that each irreducible factor is
	normalized so that its long roots have squared length $2$.
	Indeed, the algebraic arguments already allow reducible
	systems. In the geometric reconstruction, the possible squared
	root lengths are still $2$, $1$, and $2/3$, so edge lengths
	still determine root lengths and widths separately.  The identification of the facets does not require a connected
	Dynkin diagram, the sign corrections are already performed
	separately on the connected components of the support, and
	the cutting-wall depths are recovered from distances by the
	same formula.
\end{remark}

\begin{theorem}[Dream 2]\label{cor:class}
Let $P=P^{\Phi}_{\mu,\lambda}$ and $P'=P^{\Phi'}_{\mu',\lambda'}$ be realized polytopes with
nonempty supports $I, I'$. Then $P$ and $P'$ are isometric if and only if there are a bijection
$\pi:I\to I'$ and a based isometry $\sigma:\Phi_I\to\Phi'_{I'}$, $\alpha_i\mapsto\alpha'_{\pi i}$,
such that
\[
m'_{\pi i}=m_i\quad(i\in I),\qquad \pi(\mathrm{Cut})=\mathrm{Cut}',\qquad
\delta'_{\pi i}=\delta_i\quad(i\in\mathrm{Cut}).
\]
The raw shape of a realized polytope is
therefore classified, up to isometry, by the data $(\Phi_I,M,\delta|_{\mathrm{Cut}})$ up to based
isometry.
\end{theorem}

\begin{proof}

	One direction is Proposition~\ref{prop:recover}. For the converse, we show that
	the affine isometry induced by $\sigma$ matches first the two parallelepipeds
	and then their cutting half-spaces.

Write $\sigma$ also for its linear extension $U\to U'$, and define
	the affine isometry
	$$
	\widehat\varphi:\mu+U\longrightarrow\mu'+U',
	\qquad
	\mu+u\longmapsto\mu'+\sigma(u).
	$$
By the definition of $\widehat\varphi$, we have
$$
\widehat\varphi(\Pi)
=\mu'+\sum_{i\in I}[0,m_i]\sigma(\alpha_i)
=\mu'+\sum_{i\in I}[0,m'_{\pi i}]\alpha'_{\pi i}
=\Pi'.
$$

	It remains to show that $\widehat\varphi$ matches the cutting half-spaces.
Fix $i\in\mathrm{Cut}$. We will first prove that 
	\[\tau'_{\pi i}\circ\widehat\varphi=\tau_i: \mu+U\rightarrow \mathbb{R}.\] First observe that these
	two affine functions have the same linear part. Indeed, for
	$u\in U$, the linear part of $\tau'_{\pi i}\circ\widehat\varphi$ sends
	$u$ to
	\[
	\frac{(\sigma(u),\alpha'_{\pi i})}{D'_{\pi i}}
	=
	\frac{(\sigma(u),\sigma(\alpha_i))}{D_i}
	=
	\frac{(u,\alpha_i)}{D_i},
	\]
	where we used $D'_{\pi i}=D_i$ and the fact that $\sigma$ is an isometry.
	By \eqref{sigmai}, this is precisely the linear part of $\tau_i$.
	Therefore the two functions differ by a constant.  	
	
	To determine this constant, we compare their minima on $\Pi$. Since
	$\widehat\varphi(\Pi)=\Pi'$ and $\delta'_{\pi i}=\delta_i$, we have
	$\min_{\Pi}(\tau'_{\pi i}\circ\widehat\varphi)
	=\min_{\Pi'}\tau'_{\pi i}
	=-\delta'_{\pi i}
	=-\delta_i
	=\min_{\Pi}\tau_i$.
	Thus the constant is zero, proving that
	$\tau'_{\pi i}\circ\widehat\varphi=\tau_i$. 	
	
	Thus a point $v\in\Pi$ satisfies $\tau_i(v)\ge 0$ if and only if
	its image $\widehat\varphi(v)$ satisfies
	$\tau'_{\pi i}(\widehat\varphi(v))\ge 0$.
	Since $\pi$ maps $\mathrm{Cut}$ bijectively onto $\mathrm{Cut}'$,
	$v$ satisfies all the cutting-wall inequalities defining $P$
	if and only if $\widehat\varphi(v)$ satisfies those defining $P'$.
	Together with $\widehat\varphi(\Pi)=\Pi'$, this gives
	$\widehat\varphi(P)=P'$ by \eqref{Peqs}.

\end{proof}

	\begin{remark}\label{rem:isometries}
			Suppose that $\Phi=\Phi'$ and both supports are full.
		After composing with the cut-free component reversals of Proposition \ref{prop:recover}, part~\ref{perm},
		the linear part sends $\alpha_i$ to $\alpha_{\pi i}$ for a Dynkin
		diagram automorphism $\pi$. The resulting affine isometry sends
		$\Pi$ onto $\Pi'$ and therefore $\mu$ to $\mu'$: these are the unique
		vertices whose outgoing edges are positive multiples of the simple roots.
		Thus the corrected isometry has the form
		$$
		\mu+\sum_i a_i\alpha_i
		\longmapsto
		\mu'+\sum_i a_i\alpha_{\pi i}.
		$$
		
		In particular, if $\Phi$ is irreducible and its Dynkin diagram has
		no nontrivial automorphism, as in types $A_1$, $B_r$, $C_r$, $F_4$,
		$G_2$, $E_7$ and $E_8$, the corrected isometry is the translation
		$v\mapsto v+\mu'-\mu$.
		If moreover some wall cuts $\Pi$, no component reversal is available, so this
		translation is the only isometry $P\to P'$.
		If no wall cuts $\Pi$, there is exactly one additional isometry,
		obtained by composing the translation with the central inversion
		$v\mapsto\mu+\lambda-v$.
		
		Nontrivial diagram automorphisms give further possibilities.
		For example, in type $A_2$, the pairs
		$(\mu,\lambda)=(0,3\varpi_1)$ and
		$(\mu',\lambda')=(0,3\varpi_2)$ give pentagons exchanged by the
		reflection $\alpha_1\leftrightarrow\alpha_2$.
	
	\end{remark}

We now construct $K_P$ directly from the metric space $P$.
Choose an isometric embedding of $P$ into a Euclidean space and
identify $P$ with its image.  All constructions below take place in the affine hull of the
	chosen image of $P$. Any two isometric embeddings of $P$ are
	related by an affine isometry of their affine hulls, which
	transports each step of the construction to the corresponding
	step in the other embedding,
so the choice of embedding does not affect the result. 
Call a facet \emph{paired} if its outer normal has an antipodal partner
among the facet normals of $P$, and \emph{unpaired} otherwise.

\begin{enumerate}
	\item Let $\Pi_P$ be the intersection of the closed supporting
	half-spaces of the paired facets, each chosen to contain $P$.
	
	\item Choose a vertex $v$ of $\Pi_P$ whose outgoing edge vectors
	$e_1,\dots,e_d$, where $d:=\dim P$, satisfy $(e_i,e_j)\le0$ for $i\ne j$, and such that
	every unpaired outer normal is a negative multiple of some $e_i$.
	
	\item Write
	$$
	\lVert e_i\rVert^2=m_i^2c_i,
	\qquad m_i\in\Z_{>0},\quad c_i\in\{2,1,\tfrac23\},
	$$
	and set $\tilde\alpha_i:=e_i/m_i$ and $M:=(m_1,\dots,m_d)$.
	
	\item If an unpaired facet $F$ has outer normal a positive multiple
	of $-\tilde\alpha_i$, choose a point $x\in\Pi_P$ minimizing
	$(x,\tilde\alpha_i)$ and set
	$$
	\delta_i^P\coloneqq
	\frac{2\,\operatorname{dist}(x,\operatorname{aff}F)}
	{\lVert\tilde\alpha_i\rVert}.
	$$
	Otherwise, set $\delta_i^P:=-\infty$.
\end{enumerate}
The theorem below shows that the construction is well defined:
a vertex satisfying the conditions in step~(2) exists, the
decomposition in step~(3) exists and is unique, and the depth
in step~(4) is independent of the minimizing point chosen.
Moreover, the output is independent of the vertex chosen in
step~(2), the ordering of the edges, and the isometric embedding,
up to a based isometry matching the widths and reconstructed
depths. The theorem also shows that the roots $\tilde\alpha_i$
form a base of a root system $\Psi_P$. Using the reconstructed data \(M\) and \(\delta^P\),
we define
\[
K_P(q):=\Kf_{\Psi_P}(M,\delta^P)(q),
\]
where $\Kf_{\Psi_P}(M,\delta^P)$ is evaluated through \eqref{eq:KP} and \eqref{eq:vw}, with $J^P=\{i:\delta_i^P\ge m_i\}$ in place of $J$: for $w\in W_{J^P}$ and any reduced expression $s_{j_1}\cdots s_{j_r}$ of $w$, we put $\tilde v_w:=\tilde\beta-\sum_{k}(2m_{j_k}-\delta^P_{j_k})\,u_k(\tilde\alpha_{j_k})$, where $\tilde\beta=\sum_i m_i\tilde\alpha_i$ and $\Psi_P$ is the root system with base $\{\tilde\alpha_i\}$; independence of the reduced expression is part of Theorem~\ref{prop:KPshape}.
Indices with \(\delta_i^P=-\infty\) are excluded from \(J^P\),
so these entries do not enter the formula. If $P$ is a single point, we set $K_P:=1$.

\begin{theorem}[Dream 1]\label{prop:KPshape}
	For every realized polytope $P=P^\Phi_{\mu,\lambda}$, the construction above is well defined. Its output
	$(\Psi_P,M,$$\delta^P$$)$ is independent of the choices up to a based
	isometry matching $M$ and $\delta^P$, and
	$$
	K_P(q)=K^\Phi_{\lambda\mu}(q).
	$$
\end{theorem}

\begin{proof}
	If $P$ is a point, then $\mu=\lambda$ and $K_P=1=K^\Phi_{\lambda\mu}$; assume from now on that $\dim P>0$.
	We show that the construction recovers the original widths and
	cutting-wall depths. The only ambiguity is the choice of vertex
	in step~(2): on each connected component with no cutting wall,
	there are two possible choices, which reverse all the simple
	roots of that component simultaneously.
	To distinguish the reconstructed depths from the original ones, we write
	$\delta^P$ for the vector produced in step~(4), and retain $\delta$
	for the depths of the realization. 
	
	A change of isometric embedding transports every step of the
	construction. We may therefore work in the embedding
	$P\subseteq\mu+U$ supplied by the realization.
	By Lemmas~\ref{lem:dichotomy} and~\ref{lem:inventory}, the paired
	facets are exactly the parallelepiped facets, while the unpaired
	facets are the wall facets
	$P\cap\{\tau_i=0\}$, $i\in\mathrm{Cut}$, with outer normals
	$-\alpha_i$. Hence step~(1) recovers
	\[
	\Pi_P=\Pi.
	\]

	We now describe all the choices allowed in step~(2).
	The vertices of $\Pi$ are
	\[
	v_S\coloneqq\mu+\sum_{i\in S}m_i\alpha_i,
	\qquad \mathrm{with}\ S\subseteq I.
	\]

For each \(i\in I\), let \(e_i^S\) be the vector corresponding to
the edge of \(\Pi\) containing \(v_S\) and parallel to \(\alpha_i\),
directed away from \(v_S\). Then
\[
e_i^S=\varepsilon_i m_i\alpha_i,
\qquad
\varepsilon_i:=
\begin{cases}
	-1,& i\in S,\\
	+1,& i\notin S.
\end{cases}
\]
For distinct \(i,j\in I\), we have
\[
(e_i^S,e_j^S)
=\varepsilon_i\varepsilon_j m_i m_j(\alpha_i,\alpha_j).
\]

	The inner product $(\alpha_i,\alpha_j)$ is negative when $i,j$ are adjacent
	and zero otherwise. Thus the angle condition in step~(2) is
	equivalent to requiring that $\varepsilon_i$ be constant on each
	connected component of $I$. 
For \(j\in\mathrm{Cut}\), the corresponding unpaired outer normal is
\(-\alpha_j\). Since no other simple root is parallel to \(\alpha_j\),
this normal is a negative multiple of an outgoing edge vector precisely
when \(e_j^S=m_j\alpha_j\), that is, when \(\varepsilon_j=+1\).  Consequently, the vertices admissible in step~(2) are precisely
those whose sign assignments are constant on each connected component
and equal to \(+1\) on every component containing a cutting wall.
	In particular, $v=\mu$ is admissible. 
	
		We now compare the data recovered from an admissible vertex with
	those of the original realization. Since \(e_i^S=\varepsilon_i m_i\alpha_i\), we have
	\[
	\|e_i^S\|^2=m_i^2\|\alpha_i\|^2,
	\qquad
	\|\alpha_i\|^2\in\{2,1,2/3\}.
	\]

	The ratio of two distinct elements of \(\{2,1,2/3\}\) is never
	the square of a rational number. Then this decomposition is unique:
	step~(3) recovers \(m_i\) and gives
	\(\widetilde\alpha_i=\varepsilon_i\alpha_i\).

	To describe this simultaneous change of signs, define
	\[
	\sigma_\varepsilon:U\longrightarrow U,
	\qquad
	\sigma_\varepsilon(\alpha_i)=\varepsilon_i\alpha_i.
	\]
	Since the signs are constant on each connected component,
	\(\sigma_\varepsilon\) acts as either \(\mathrm{id}\) or
	\(-\mathrm{id}\) on each component subspace $U_c$. These subspaces are
	mutually orthogonal, so \(\sigma_\varepsilon\) is a linear isometry.
	Moreover, it fixes \(\alpha_i\) for every \(i\in\mathrm{Cut}\). 
Thus the roots recovered in step~(3) form the base
\(\{\widetilde\alpha_i\}=\sigma_\varepsilon(\{\alpha_i\})\)
of \(\Psi_P:=\sigma_\varepsilon(\Phi_I)\), and
\(\sigma_\varepsilon\) is a based isometry from \(\Phi_I\)
to \(\Psi_P\). In fact, \(\Psi_P=\Phi_I\) as a set of roots, since
\(\sigma_\varepsilon\) acts as \(\pm\mathrm{id}\) on each
component, what can change is the chosen base.

We now compute the depths \(\delta_i^P\) defined in step~(4).
An unpaired facet \(F\) with outer normal a positive multiple
of \(-\tilde\alpha_i\) exists exactly when \(i\in\mathrm{Cut}\).
In that case, \(\tilde\alpha_i=\alpha_i\) and
\(\operatorname{aff}F=\{\tau_i=0\}\).
Every point \(x\in\Pi\) minimizing \((x,\alpha_i)\) lies on
\(\{\tau_i=-\delta_i\}\).  By the depth calculation in
Proposition~\ref{prop:recover}, the distance between these
parallel hyperplanes is \(\delta_i\|\alpha_i\|/2\).
Thus the value in step~(4) is independent of the minimizing
point \(x\), and
\[
\delta_i^P=
\begin{cases}
	\delta_i,& i\in\mathrm{Cut},\\
	-\infty,& i\notin\mathrm{Cut}.
\end{cases}
\]

We have now verified that every step of the construction is
well defined: step~(1) recovers $\Pi$, we have classified all the admissible choices of vertices in step~(2), and in particular  the vertex $\mu$ is admissible, for any admissible vertex the decomposition in step~(3) exists
and is unique, and step~(4) is independent of the minimizing
point chosen. We have also shown that the recovered roots form
a base of $\Psi_P$.

It remains to check that the choices made in the construction
do not affect the resulting data, up to a based isometry.
For two vertices satisfying step~(2), let $\varepsilon$ and
$\varepsilon'$ be their sign assignments. The corresponding
widths and reconstructed depths are the same, and the based
isometry
\[
\sigma_{\varepsilon'}\sigma_\varepsilon^{-1}
\]
sends $\varepsilon_i\alpha_i$ to $\varepsilon_i'\alpha_i$
for every $i\in I$. Thus it matches the two bases together
with their widths and reconstructed depths. Reordering the
edges merely relabels these data, while changing the isometric
embedding transports every step by an affine isometry.
This proves the asserted independence of choices.

We now show that the recovered data give the required
polynomial. By Theorem~\ref{thm:parab}(1), we only need to
compare the formula defining $K_P$ with the formula for
the restricted pair $(\mu_I,\lambda_I)$ in $\Phi_I$.
The latter has widths $M$ and depths $\delta|_I$, as shown
in the proof of that theorem. There are therefore only two
differences to account for: the replacement of $\delta|_I$
by $\delta^P$, and the change of base from $\{\alpha_i\}$
to $\{\tilde\alpha_i\}$.

First, the replacement of the depths discards only entries
that do not contribute to the formula. Indeed, the depths
agree on $\mathrm{Cut}$, and
$J\subseteq\mathrm{Cut}$. Outside $\mathrm{Cut}$, the original
depths do not contribute to the active set, and neither do
the reconstructed values $-\infty$.  Thus, 
\[
J^P=\{i\in I:\delta_i^P\ge m_i\}
=
\{i\in I:\delta_i\ge m_i\}
=J.
\]
Moreover, for every $j\in J$,
\[
2m_j-\delta_j^P=2m_j-\delta_j=\hat\lambda_j.
\]
Thus the active set and all coefficients used
in~\eqref{eq:vw} are unchanged.

Second, changing the base does not change the simple
reflections, since
\[
s_{\tilde\alpha_i}=s_{\varepsilon_i\alpha_i}=s_{\alpha_i}.
\]
Thus both formulas are indexed by the same subgroup $W_J$,
with the same length function. What may change is the set
of positive roots, and hence the partition function.

To compare the vectors entering these partition functions,
note that $\sigma_\varepsilon$ commutes with every simple
reflection and sends
\[
\beta=\sum_{i\in I}m_i\alpha_i
\quad\text{to}\quad
\tilde\beta=\sum_{i\in I}m_i\tilde\alpha_i.
\]
Together with the equality of coefficients proved above,
formula~\eqref{eq:vw} therefore gives, for every $w\in W_J$, 
\[
\tilde v_w=\sigma_\varepsilon(v_w),
\]
where $\tilde v_w$ is computed from the reconstructed data
and $v_w$ from the restricted pair. This also proves that
$\tilde v_w$ is independent of the reduced expression used
to compute it, since $v_w$ is.

Finally, $\sigma_\varepsilon$ carries the positive roots
for the original base to those for the reconstructed base.
By Lemma~\ref{lem:basediso},
\[
\mathcal P_q^{\Psi_P}(\tilde v_w)
=
\mathcal P_q^{\Phi_I}(v_w).
\]
Since the indexing set and the signs $(-1)^{\ell(w)}$ are
also unchanged, the two sums agree term by term.

This comparison proves that the formula defining $K_P$
is well defined and yields
\[
K_P(q)
=
\mathcal K_{\Phi_I}(M|_I,\delta|_I)(q)
=
K^{\Phi_I}_{\lambda_I\mu_I}(q)
=
K^\Phi_{\lambda\mu}(q),
\]
where the last two equalities follow from
Proposition~\ref{prop:KP} and Theorem~\ref{thm:parab}(1).
In particular, every choice in the construction yields
the same polynomial, as required.

\end{proof}

\section{A conjecture and open problems}\label{sec:hecke}

 In this section we speculate about new lines of research opened by this work.

\subsection{Isometry versus interval isomorphism}
Let $\Phi$ be an irreducible root system. For dominant coweights $\mu\le\lambda$ of $\Phi$ we consider the polytope
$P^{\Phi^\vee}_{\mu,\lambda}=(\mu+C)\cap(\lambda-C)\cap(\overline{D}-\rho^\vee)$, that is, the definition of Section~\ref{sec:intro} applied to $\Phi^\vee$, normalized so that long coroots have length $\sqrt2$. We use the conventions of \cite{BOAT}: the fundamental alcove is the antidominant one, $$A_{\mathrm{id}}=\{v:\,-1<( v,\alpha)<0 \text{ for all } \alpha\in\Phi^+\}.$$
  Each $x\in W_{\mathrm{aff}}$ corresponds to the alcove $A_x=x(A_{\mathrm{id}})$, and $\theta(\mu)\in W_{\mathrm{aff}}$ is defined by $A_{\theta(\mu)}=A_{w_0}+\mu$. By the Lusztig--Kato theorem, $K^{\Phi^\vee}_{\lambda\mu}$ is, after a change of normalization, the Kazhdan--Lusztig polynomial $h_{\theta(\mu),\theta(\lambda)}$.

\begin{conjecture}\label{conj:intervals}
Let $\Phi,\Phi'$ be irreducible root systems, $\mu\le\lambda$ dominant coweights of $\Phi$ and $\mu'\le\lambda'$ dominant coweights of $\Phi'$. If $P^{\Phi^\vee}_{\mu,\lambda}$ and $P^{\Phi'^\vee}_{\mu',\lambda'}$ are isometric, then the Bruhat intervals $[\theta(\mu),\theta(\lambda)]\subseteq W_{\mathrm{aff}}(\Phi)$ and $[\theta'(\mu'),\theta'(\lambda')]\subseteq W_{\mathrm{aff}}(\Phi')$ are isomorphic as posets.
\end{conjecture}

Together with the combinatorial invariance conjecture (a proof of which has recently been announced in \cite{OpenAICIC}), Conjecture~\ref{conj:intervals} would imply Theorem~\ref{thm:main}. We tested it in about $64\,000$ comparisons within one type, in affine types of rank at most $6$, and in more than $9\,000$ comparisons between different Dynkin types, with intervals of up to $263\,448$ elements; in each case we exhibited an explicit poset isomorphism.

The converse fails. If $\lambda-\mu=\alpha_i^\vee$ and no wall cuts the polytope, then $P^{\Phi^\vee}_{\mu,\lambda}$ is a segment of length $\lVert\alpha_i^\vee\rVert$, while $[\theta(\mu),\theta(\lambda)]$ has length two and is therefore the diamond with rank sizes $1,2,1$. For instance, in type $B_2$ with $\mu=2\varpi^\vee_1+2\varpi^\vee_2$, the pairs $(\mu,\mu+\alpha_1^\vee)$ and $(\mu,\mu+\alpha_2^\vee)$ give segments of lengths $1$ and $\sqrt2$ and isomorphic intervals. Thus the metric shape of the polytope contains information that the interval does not see.

\subsection{Towards a positive formula}

Lusztig's formula \eqref{eq:lusztig} is an alternating sum, and for general $\Phi$
no uniform positive combinatorial formula for $K^\Phi_{\lambda\mu}$ is known.
In type $A$, the classical charge statistic of Lascoux and Sch\"utzenberger
\cite{LS} gives such a formula; more recently, Plaza and Sagurie obtained a
different combinatorial algorithm  \cite{PlazaSagurie} through the pre-canonical bases introduced in
\cite{LPP}. Beyond type $A$, Patimo and Torres constructed
 a charge statistic in type $C_2$, giving a positive
combinatorial formula for the corresponding Kostka--Foulkes polynomials
\cite{PatimoTorres}.  More
generally, in type $C$, combinatorial formulas have been obtained from energy
functions on Kirillov--Reshetikhin crystals \cite{CKL}. In type $B_2$, Batistelli, Bingham, and Plaza \cite{BBP},  and in type $G_2$, Muniz, Plaza and Rojas-And\'ias \cite{MPR} found combinatorial formulas for generalized Kostka--Foulkes polynomials.

The results of this paper suggest where a uniform positive formula should live.
Proposition~\ref{prop:KP} and Theorem~\ref{thm:parab} express
$K^{\Phi}_{\lambda\mu}$ entirely in terms of the data $(\Phi_I,M,\delta|_{\mathrm{Cut}})$, while
Theorem~\ref{cor:class} shows that this data, up to based isometry, is determined by the isometry
class of $P^{\Phi}_{\mu,\lambda}$. Thus a positive formula built directly from $(\Phi_I,M,\delta|_{\mathrm{Cut}})$ would
make Theorem~\ref{thm:main} automatic: the objects being counted would depend only on the shape,
and not on the realization that produced it.

One possible route is already suggested by Lusztig's formula itself. In each degree, its terms
count lattice points in polytopes, with the alternating signs separating them into positive and
negative families. It is therefore natural to ask whether the cancellation can be made explicit,
by matching the negative lattice points with positive ones and leaving a set of survivors whose
cardinality is $[q^k]K^{\Phi}_{\lambda\mu}$. We have evidence for such a picture in low rank.

\subsection{Towards a geometric explanation}
 A more basic question remains. The polytope $P^\Phi_{\mu,\lambda}$ is an elementary convex body, defined with no reference to sheaves or filtrations, yet its isometry class determines a polynomial encoding local intersection cohomology on the affine Grassmannian
 \cite[\S11]{Lusztig}\cite[\S\S2.4--2.6]{BravermanFinkelberg}.
We do not know a geometric explanation for this phenomenon. Understanding why the Euclidean geometry of the polytope remembers this intersection cohomology
is an open problem.

\end{document}